\documentclass[11pt,amssymb,amsfont,a4paper]{article}
\usepackage{latexsym,amssymb,amsmath,amsthm,color}
\usepackage{geometry}
\usepackage{url}
\usepackage[utf8]{inputenc}
\usepackage{authblk}

\theoremstyle{definition}
\newtheorem{definition}{Definition}
\newtheorem{example}{Example}

\theoremstyle{plain}
\newtheorem{theorem}{Theorem}
\newtheorem{proposition}{Proposition}
\newtheorem{lemma}{Lemma}
\newtheorem{remark}{Remark}
\newtheorem{corollary}{Corollary}

\title{Classification of multivariate skew polynomial rings over finite fields via affine transformations of variables}
\author[1]{Umberto Mart{\'i}nez-Pe\~{n}as \thanks{umberto@ece.utoronto.ca}}
\affil[1]{Dept.\ of Electrical \& Computer Engineering, University of Toronto, Canada}

\date{}

\begin{document}

\maketitle

\begin{abstract}
In this work, free multivariate skew polynomial rings are considered, together with their quotients over ideals of skew polynomials that vanish at every point (which includes minimal multivariate skew polynomial rings). We provide a full classification of such multivariate skew polynomial rings (free or not) over finite fields. To that end, we first show that all ring morphisms from the field to the ring of square matrices are diagonalizable, and that the corresponding derivations are all inner derivations. Secondly, we show that all such multivariate skew polynomial rings over finite fields are isomorphic as algebras to a multivariate skew polynomial ring whose ring morphism from the field to the ring of square matrices is diagonal, and whose derivation is the zero derivation. Furthermore, we prove that two such representations only differ in a permutation of the field automorphisms appearing in the corresponding diagonal. The algebra isomorphisms are given by affine transformations of variables and preserve evaluations and degrees. In addition, ours proofs show that the simplified form of multivariate skew polynomial rings can be found computationally and explicitly.

\textbf{Keywords:} Affine transformations, derivations, free polynomial rings, Moore matrices, multivariate skew polynomial rings, Vandermonde matrices.

\textbf{MSC:} 11T06, 11T30, 12E10, 12E20. 
\end{abstract}

\section{Introduction} \label{sec intro}

\textit{Univariate skew polynomial rings} over a division ring $ \mathbb{F} $ (or any ring in general) were formally introduced in \cite{ore} and further studied over finite fields $ \mathbb{F}_q $ in \cite{orespecial}, although some explicit applications of this latter case can be found much earlier \cite{moore}. These rings are ``non-commutative polynomial rings'' $ \mathcal{R} $ in one variable $ x $ (formally, $ \mathcal{R} $ is a left vector space over $ \mathbb{F} $ with infinite left basis $ \{ x^i \mid i = 0,1,2, \ldots \} $) such that addition is as usual and multiplication satisfies the following: $ \mathcal{R} $ is a (commutative or non-commutative) ring with identity $ 1 = x^0 $, where the $ i $th power of $ x = x^1 $ corresponds to the monomial $ x^i $ and the degree of a product of two polynomials is the sum of their degrees. 

Such algebraic structures on the left vector space $ \mathcal{R} $ with left basis $ \{ x^i \mid i = 0,1,2, \ldots \} $ correspond bijectively with pairs $ (\sigma, \delta) $, where $ \sigma : \mathbb{F} \longrightarrow \mathbb{F} $ is a ring endomorphism and $ \delta : \mathbb{F} \longrightarrow \mathbb{F} $ is a $ \sigma $-derivation. In other words, $ \delta $ is additive and, for all $ a,b \in \mathbb{F} $, it holds that $ \delta(ab) = \sigma(a) \delta(b) + \delta(a)b $, for all $ a,b \in \mathbb{F} $. This bijection is due to the relation $ x a = \sigma(a) x + \delta(a) $, for all $ a \in \mathbb{F} $, that needs to happen in $ \mathcal{R} $ in order to satisfy the properties mentioned above. See \cite[pp. 481--482]{ore} for more details.

Since univariate skew polynomial rings are right Euclidean domains \cite[p. 483]{ore}, the \textit{evaluation} of a skew polynomial $ F(x) \in \mathcal{R} $ over a point $ a \in \mathbb{F} $ can be defined as the remainder of the Euclidean division of $ F(x) $ by $ x - a $ on the right. This concept of evaluation was first considered in full generality by Lam and Leroy in \cite{lam, lam-leroy}. Among others, this notion of evaluation helps unify the study of Vandermonde, Moore and Wronskian matrices \cite{lam, lam-leroy} and more general matrix types (see \cite[p. 604]{linearizedRS} for instance), and gives a natural framework for Hilbert 90 Theorems \cite{hilbert90} and pseudolinear transformations \cite{leroy-pol}. Specifically over finite fields, this notion has provided maximum rank distance (error-correcting) codes \cite{gabidulin} (based on Moore matrices \cite{moore}), and maximum sum-rank distance codes \cite{linearizedRS} with finite-field sizes that are not exponential in the code length, in constrast with \cite{gabidulin} (see \cite[Sec. 4.2]{linearizedRS}). 

Recently, two natural extensions of this concept of evaluation have been given for multivariate skew polynomials. In \cite{skewRM}, it is proposed to evaluate \textit{certain} iterated skew polynomials \cite[Sec. 8.8]{cohn} over \textit{certain} affine points by a multivariate Euclidean division algorithm. This lack of \textit{universality} is due to the lack of unique remainders when dividing an iterated skew polynomial by linear skew polynomials $ x_1 - a_1 $, $ x_2 - a_2 $, $ \ldots $, $ x_n - a_n $, which is due to the relations between variables in iterated skew polynomial rings with $ n > 1 $ variables (see \cite[Ex. 3.5]{skewRM} and \cite[Remarks 7 \& 8]{multivariateskew}). Alternatively, a natural and universal definition of evaluation was given in \cite[Def. 9]{multivariateskew} for \textit{free multivariate skew polynomials}, precisely due to the lack of relations among the variables. Due to \cite[Lemma 5]{multivariateskew}, \textit{any} free multivariate skew polynomial can be evaluated over \textit{any} affine point as the \textit{unique} remainder of the division by $ x_1 - a_1 $, $ x_2 - a_2 $, $ \ldots $, $ x_n - a_n $ on the right. Multivariate skew polynomial rings with relations on the variables and where evaluation is still universal can then by defined as quotients of the free ring by two-sided ideals of skew polynomials that vanish at every point \cite[Def. 19]{multivariateskew}. However, these rings are not iterated skew polynomial rings in general.

Thanks to the definitions considered in \cite{multivariateskew}, the central results from \cite{lam, lam-leroy} on the roots of skew polynomials were translated from the univariate to the multivariate case in \cite{multivariateskew}. Among these results, the matroidal or lattice structure of the sets of roots of univariate skew polynomials played a crucial role in the properties of the matrices of Vandermonde type considered in \cite{lam, lam-leroy, linearizedRS} and of the linear codes in \cite{gabidulin, linearizedRS}, especially when considered over finite fields. Apart from its own interest, one of the main motivations behind \cite{skewRM, multivariateskew} is to define linear codes of Reed-Muller type \cite{muller, reed} in the form of evaluation codes that would share the features of the codes in \cite{gabidulin, linearizedRS}, but allowing more general parameters and properties. As a collateral consequence, it was recently shown in \cite{lin-multivariateskew} that Hilbert's Theorem 90 over general Galois extensions of fields (as proven originally by Noether \cite{noether}, see also \cite[Th. 21]{artin-lectures}) can be obtained from such a structure of the roots of multivariate skew polynomials.

To understand the extent of such applications, it is of interest to classify multivariate skew polynomial rings and to find explicit descriptions of the ring morphisms $ \sigma : \mathbb{F} \longrightarrow \mathbb{F}^{n \times n} $ and the $ \sigma $-derivations $ \delta : \mathbb{F} \longrightarrow \mathbb{F}^n $ that define such rings (see \cite[Defs. 1 \& 2]{multivariateskew} or Section \ref{sec preliminaries} below). In the univariate case over finite fields, it is well-known that all endomorphisms are of Frobenius type \cite[Th. 2.21]{lidl} and all derivations are inner derivations \cite[Sec. 8.3]{cohn} (see also \cite[Sec. 4.2]{linearizedRS} or Propositions \ref{prop description morphisms univariate} and \ref{prop description non-vect derivations} below). To the best of our knowledge, explicit descriptions are not yet known when $ n > 1 $. 

The first main objective and results in this paper are showing that, in the multivariate case over finite fields, all morphisms are \textit{diagonalizable} (Theorem \ref{theorem description matrix morph} in Subsection \ref{subsec morphisms}) and all derivations are \textit{inner derivations} (Theorem \ref{theorem description vector der} in Subsection \ref{subsec derivations}), providing explicit descriptions of such objects. As in the univariate case, we believe that such characterizations of morphisms and derivations over finite fields are of interest by themselves. Furthermore, these characterizations do not hold in general for infinite fields (see Examples \ref{ex non inner derivation} and \ref{ex non diagonalizable morphism}). Moreover, the well-known Skolem-Noether theorem \cite{skolem} cannot be directly applied to characterize ring morphisms $ \sigma : \mathbb{F} \longrightarrow \mathbb{F}^{n \times n} $ over finite fields (see Subsection \ref{subsec morphisms}).

The second main objective and results of this paper are showing that the given characterizations imply that any multivariate skew polynomial ring over a finite field is isomorphic as an algebra to a multivariate skew polynomial ring given by a diagonal morphism $ \sigma = {\rm diag}(\sigma_1, \sigma_2, \ldots, \sigma_n) $ and a zero derivation $ \delta = 0 $ (Theorem \ref{th reduction finite fields} in Section \ref{sec classification finite fields}). Such a ring isomorphism is given by affine transformations of variables (meaning a composition of a linear tranformation and a translation), which moreover preserve evaluations and degrees (see Sections \ref{sec linear transformations} and \ref{sec translations} for linear transformations and translations, respectively). In addition, we show that affine transformations of variables are the only left $ \mathbb{F} $-algebra isomorphisms with these properties (Theorem \ref{th affine transformations} Section \ref{sec affine}). When $ \mathbb{F} $ is finite, this implies that two multivariate skew polynomial rings given by diagonal ring morphisms $ \sigma = {\rm diag}(\sigma_1, \sigma_2, \ldots, \sigma_n) $ and $ \tau = {\rm diag}(\tau_1, \tau_2, \ldots, \tau_n) $ are isomorphic by such an algebra isomorphism if, and only if, $ (\sigma_1, \sigma_2, \ldots, \sigma_n) $ is a permutation of $ (\tau_1, \tau_2, \ldots, \tau_n) $ (Theorem \ref{th classification} in Section \ref{sec classification finite fields}). We believe that the study of affine transformations of variables is of interest by itself, and thus we consider general division rings $ \mathbb{F} $ for such results. 

In conclusion, these results provide a full and explicit classification of free and non-free multivariate skew polynomial rings over finite fields, taking evaluations and degrees into account (Section \ref{sec classification finite fields}). Furthermore, the proofs give explicit and computational methods to find the reduced form of a multivariate skew polynomial ring.

\section*{Notation}

Unless otherwise stated, $ \mathbb{F} $ will denote a division ring. A field is a commutative division ring. All rings in this work will be assumed to have (multiplicative) identity and all ring morphisms are required to map identities to identities. Throughout this work, $ q $ is a fixed power of a prime number $ p $, and $ \mathbb{F}_q $ denotes the finite field with $ q $ elements. For positive integers $ m $ and $ n $, $ \mathbb{F}^{m \times n} $ will denote the set of $ m \times n $ matrices over $ \mathbb{F} $, whose identity (if $ m = n $) is denoted by $ I $, and $ \mathbb{F}^n $ will denote the set of column vectors of length $ n $ over $ \mathbb{F} $. That is, $ \mathbb{F}^n = \mathbb{F}^{n \times 1} $.

\section{Preliminaries: Multivariate skew polynomial rings} \label{sec preliminaries}

We will consider multivariate skew polynomial rings as in \cite{multivariateskew} for the reasons stated in the Introduction. Let $ x_1, x_2, \ldots, x_n $ be pair-wise distinct letters, which we will call \textit{variables}, and let $ \mathcal{M} $ be the free monoid with basis $ x_1, x_2, \ldots, x_n $ \cite[Sec. 6.5]{cohn}, whose elements are called \textit{monomials} and whose identity is denoted by $ 1 $. We will denote $ \mathbf{x} = (x_1, x_2, \ldots, x_n)^T \in \mathcal{M}^n $ and we will typically denote monomials in such variables by $ \mathfrak{m}(\mathbf{x}) $, or $ \mathfrak{m} $ for brevity if there is no confusion about $ \mathbf{x} $. We define the \textit{degree} of a given monomial $ \mathfrak{m}(\mathbf{x}) \in \mathcal{M} $, denoted by $ \deg(\mathfrak{m}(\mathbf{x})) $, as its length as a string in $ x_1, x_2, \ldots, x_n $, where $ \deg(1) = 0 $.

Let $ \mathcal{R} $ be the left vector space over $ \mathbb{F} $ with basis $ \mathcal{M} $. We call its elements \textit{free multivariate skew polynomials}, or simply skew polynomials. They are of the form
$$ F(\mathbf{x}) = \sum_{\mathfrak{m}(\mathbf{x}) \in \mathcal{M}} F_\mathfrak{m} \mathfrak{m}(\mathbf{x}), $$
where $ F_\mathfrak{m} \in \mathbb{F} $ are all zero except for a finite number of them. We define the \textit{degree} of $ F(\mathbf{x}) $, denoted by $ \deg(F(\mathbf{x})) $, as the maximum degree of a monomial $ \deg(\mathfrak{m}(\mathbf{x})) $ such that $ F_\mathfrak{m} \neq 0 $, in the case $ F \neq 0 $, and we define $ \deg(F(\mathbf{x})) = \infty $ if $ F(\mathbf{x}) = 0 $.

Following Ore's line of thought \cite[pp. 481--482]{ore}, it was shown in \cite[Th. 1]{multivariateskew} that a product in $ \mathcal{R} $ turns it into a ring with multiplicative identity $ 1 $, where products of monomials consist in appending them and where $ \deg(F(\mathbf{x})G(\mathbf{x})) = \deg(F(\mathbf{x})) + \deg(G(\mathbf{x})) $, for all $ F(\mathbf{x}),G(\mathbf{x}) \in \mathcal{R} $, if, and only if, there exist maps $ \sigma_{i,j}, \delta_i : \mathbb{F} \longrightarrow \mathbb{F} $, for $ i,j = 1,2, \ldots, n $, such that
\begin{equation}
x_i a = \sum_{j=1}^n \sigma_{i,j}(a) x_j + \delta_i(a),
\label{eq def inner product}
\end{equation}
for $ i = 1,2, \ldots, n $, and for all $ a \in \mathbb{F} $, where the map
\begin{equation*}
\sigma : \mathbb{F} \longrightarrow \mathbb{F}^{n \times n} : a \mapsto \left( \begin{array}{cccc}
\sigma_{1,1}(a) & \sigma_{1,2}(a) & \ldots & \sigma_{1,n}(a) \\
\sigma_{2,1}(a) & \sigma_{2,2}(a) & \ldots & \sigma_{2,n}(a) \\
\vdots & \vdots & \ddots & \vdots \\
\sigma_{n,1}(a) & \sigma_{n,2}(a) & \ldots & \sigma_{n,n}(a) \\
\end{array} \right)
\end{equation*}
is a \textit{ring morphism} and the map
\begin{equation*}
\delta : \mathbb{F} \longrightarrow \mathbb{F}^n : a \mapsto \left( \begin{array}{c}
\delta_1(a) \\
\delta_2(a) \\
\vdots \\
\delta_n(a)
\end{array} \right)
\end{equation*}
is a \textit{$ \sigma $-derivation}. Recall from \cite[Def. 1]{multivariateskew} that $ \delta $ is a $ \sigma $-derivation if it is additive and 
\begin{equation}
\delta(ab) = \sigma(a) \delta(b) + \delta(a) b,
\label{eq derivations multiplicative property}
\end{equation}
for all $ a,b \in \mathbb{F} $. With this compact notation, we may rewrite Equation (\ref{eq def inner product}) as
$$ \mathbf{x} a = \sigma(a) \mathbf{x} + \delta(a). $$

As in the univariate case, it was shown in \cite[Th. 1]{multivariateskew} that such pairs $ (\sigma, \delta) $ correspond bijectively with products in $ \mathcal{R} $ satisfying the properties described above. Hence we may use the notation $ \mathcal{R} = \mathbb{F}[\mathbf{x}; \sigma, \delta] $ when we consider the product in $ \mathcal{R} $ given by the pair $ (\sigma, \delta) $. In other words, $ \mathbb{F}[\mathbf{x}; \sigma, \delta] = \mathbb{F}[\mathbf{x}; \sigma^\prime, \delta^\prime] $ if, and only if, $ (\sigma, \delta) = (\sigma^\prime, \delta^\prime) $. However, it may hold that $ (\sigma, \delta) \neq (\sigma^\prime, \delta^\prime) $ and at the same time $ \mathbb{F}[\mathbf{x}; \sigma, \delta] $ and $ \mathbb{F}[\mathbf{x}; \sigma^\prime, \delta^\prime] $ are isomorphic as left $ \mathbb{F} $-algebras in a canonical way. Observe that the conventional free multivariate polynomial ring $ \mathbb{F}[\mathbf{x}] $ corresponds to $ \sigma = {\rm Id} $ and $ \delta = 0 $, where $ {\rm Id}(a) = aI $, for all $ a \in \mathbb{F} $.

Thanks to the lack of relations among the variables, it was proven in \cite[Lemma 5]{multivariateskew} that, for any $ a_1, a_2, \ldots , a_n \in \mathbb{F} $ and any $ F(\mathbf{x}) \in \mathbb{F}[\mathbf{x}; \sigma, \delta] $, there exist unique $ G_1(\mathbf{x}), G_2(\mathbf{x}), $ $ \ldots, $ $ G_n(\mathbf{x}) \in \mathbb{F}[\mathbf{x}; \sigma, \delta] $ and $ b \in \mathbb{F} $ such that
$$ F(\mathbf{x}) = \sum_{i = 1}^n G_i(\mathbf{x}) (x_i - a_i) + b. $$
Hence, we may define the $ (\sigma,\delta) $-evaluation of $ F(\mathbf{x}) $ at $ \mathbf{a} \in \mathbb{F}^n $ as the constant $ b \in \mathbb{F} $, which we will denote by
\begin{equation}
F(\mathbf{a}) = E^{\sigma,\delta}_\mathbf{a}(F(\mathbf{x})) \in \mathbb{F}.
\label{eq evaluation}
\end{equation}
The notation $ F(\mathbf{a}) $ is for ease of reading and brevity, whereas $ E^{\sigma,\delta}_\mathbf{a}(F(\mathbf{x})) $ will be necessary when we consider different pairs $ (\sigma, \delta) $.

Once the natural and universal evaluation in (\ref{eq evaluation}) is obtained over free multivariate skew polynomial rings, we may extend it to multivariate skew polynomial rings with relations on the variables as follows. Let $ I(\mathbb{F}^n) \subseteq \mathbb{F}[\mathbf{x}; \sigma, \delta] $ denote the left ideal of skew polynomials that vanish at every point. It was shown in \cite[Prop. 18]{multivariateskew} that $ I(\mathbb{F}^n) $ is a two-sided ideal. Therefore, we may define \textit{non-free multivariate skew polynomial rings}, where evaluation is still natural and universal as in (\ref{eq evaluation}), as quotient rings of the form $ \mathbb{F}[\mathbf{x}; \sigma, \delta] / I $, where $ I \subseteq I(\mathbb{F}^n) $ is a two-sided ideal. The skew polynomial ring $ \mathbb{F}[\mathbf{x}; \sigma, \delta] / I(\mathbb{F}^n) $ can be called a \textit{minimal (multivariate) skew polynomial ring}. In the conventional case $ (\sigma, \delta) = ({\rm Id}, 0) $ over a finite field $ \mathbb{F} = \mathbb{F}_q $, we recover the minimal conventional polynomial ring due to the well-known fact that
$$ I(\mathbb{F}_q^n) = \left( \{ x_i x_j - x_j x_i \mid 1 \leq i < j \leq n \} \cup \{ x_1^q - x_1, x_2^q - x_2, \ldots, x_n^q - x_n \} \right) . $$
Formally, our first main results will be showing that any pair $ (\sigma, \delta) $ is such that $ \sigma = A {\rm diag}(\sigma_1, \sigma_2, \ldots, \sigma_n) A^{-1} $ (Theorem \ref{theorem description matrix morph}) and $ \delta = A (\boldsymbol\lambda {\rm Id} - {\rm diag}(\sigma_1, \sigma_2, \ldots, \sigma_n) \boldsymbol\lambda) $ (Theorem \ref{theorem description vector der}), over finite fields $ \mathbb{F} = \mathbb{F}_q $, for an invertible matrix $ A \in \mathbb{F}_q^{n \times n} $ and a vector $ \boldsymbol\lambda \in \mathbb{F}_q^n $. Our second main results will be proving that in such cases, the free or non-free multivariate skew polynomial ring $ \mathbb{F}_q[\mathbf{x}; \sigma, \delta] $ is canonically isomorphic to the ring $ \mathbb{F}_q[\mathbf{x}; {\rm diag}(\sigma_1, \sigma_2, \ldots, \sigma_n), 0] $ (Theorem \ref{th reduction finite fields}). Note that in the latter skew polynomial ring, the commutation rule (\ref{eq def inner product}) simply reads
\begin{equation}
x_i a = \sigma_i(a) x_i,
\label{eq def inner product simplified}
\end{equation}
for all $ i = 1,2, \ldots, n $ and all $ a \in \mathbb{F}_q $. In addition, we show in Theorem \ref{th classification} that two such skew polynomial rings $ \mathbb{F}_q[\mathbf{x}; {\rm diag}(\sigma_1, \sigma_2, \ldots, \sigma_n), 0] $ and $ \mathbb{F}_q[\mathbf{x}; {\rm diag}(\tau_1, \tau_2, \ldots, \tau_n), 0] $ are isomorphic as $ \mathbb{F}_q $-algebras by an isomorphism that preserves degrees if, and only if, $ (\sigma_1, \sigma_2, \ldots, \sigma_n) $ is a permutation of $ (\tau_1, \tau_2, \ldots, \tau_n) $. This gives a full classification of free, non-free and minimal multivariate skew polynomial rings over finite fields.

\section{Explicit descriptions for finite fields} \label{sec explicit finite fields}

In the univariate case, $ n = 1 $, explicit descriptions of all possible ring morphisms $ \sigma : \mathbb{F}_q \longrightarrow \mathbb{F}_q $ and $ \sigma $-derivations $ \delta : \mathbb{F}_q \longrightarrow \mathbb{F}_q $ are well-known for any value of $ q $.

First, all ring morphisms $ \sigma : \mathbb{F}_q \longrightarrow \mathbb{F}_q $ form the Galois group of the field extension $ \mathbb{F}_p \subseteq \mathbb{F}_q $. More concretely, they are given as follows (see \cite[Th. 2.21]{lidl}).

\begin{proposition} \label{prop description morphisms univariate}
A map $ \sigma : \mathbb{F}_q \longrightarrow \mathbb{F}_q $ is a ring morphism if, and only if, there exists an integer $ 0 \leq j \leq m-1 $ such that $ \sigma(a) = a^{p^j} $, for all $ a \in \mathbb{F}_q $, where $ q = p^m $. In particular, $ \sigma $ is a field automorphism.
\end{proposition}

Second, all $ \sigma $-derivations over finite fields are \textit{inner derivations}. More concretely, they are given as follows (see \cite[Sec. 8.3]{cohn} or \cite[Sec. 4.2]{linearizedRS}).

\begin{proposition} \label{prop description non-vect derivations}
Given a field automorphism $ \sigma : \mathbb{F}_q \longrightarrow \mathbb{F}_q $, a map $ \delta : \mathbb{F}_q \longrightarrow \mathbb{F}_q $ is a $ \sigma $-derivation if, and only if, there exists an element $ \lambda \in \mathbb{F}_q $ such that 
$$ \delta(a) = \lambda (a - \sigma(a)), $$
for all $ a \in \mathbb{F}_q $. In particular, $ \delta = 0 $ is the only $ {\rm Id} $-derivation.
\end{proposition}

In the next two subsections, we will extend Propositions \ref{prop description morphisms univariate} and \ref{prop description non-vect derivations} to the case $ n > 1 $.

\subsection{All derivations are inner derivations} \label{subsec derivations}

For technical reasons, we will first give descriptions of derivations, since they will prove to be useful to describe morphisms in the next subsection. In this subsection, we will show in Proposition \ref{prop description non-vect derivations} that all derivations are also inner derivations when $ n \geq 1 $. Furthermore, we will consider a more general definition of derivations and inner derivations that will also be useful in the next subsection.

\begin{definition} \label{def vector der}
Given a field $ \mathbb{F} $ and ring morphisms $ \sigma, \tau : \mathbb{F} \longrightarrow \mathbb{F}^{n \times n} $, we say that a map $ \delta : \mathbb{F} \longrightarrow \mathbb{F}^n $ is a $ (\sigma, \tau) $-derivation if it is additive and satisfies that
\begin{equation}
\delta(ab) = \sigma(a) \delta(b) + \tau(b) \delta(a) ,
\label{eq multiplicative property vector derivations}
\end{equation}
for all $ a,b \in \mathbb{F} $. Moreover, we say that $ \delta $ is an inner $ (\sigma, \tau) $-derivation if there exists $ \boldsymbol\lambda \in \mathbb{F}^n $ such that
\begin{equation}
\delta(a) = (\tau(a) - \sigma(a)) \boldsymbol\lambda ,
\label{eq def inner derivation}
\end{equation}
for all $ a \in \mathbb{F} $.
\end{definition}

Observe that a $ \sigma $-derivation is just a $ (\sigma, {\rm Id}) $-derivation since the definition is given over a (commutative) field $ \mathbb{F} $ and hence (\ref{eq derivations multiplicative property}) holds. This is however not true over general division rings. 

In addition, if $ \mathbb{F} $ is a field and $ \sigma(a)\tau(b) = \tau(b)\sigma(a) $, for all $ a,b \in \mathbb{F} $ (which is always the case if $ \tau = {\rm Id} $ and $ \mathbb{F} $ is a field), then an inner $ (\sigma,\tau) $-derivation is indeed a $ (\sigma,\tau) $-derivation. Again, the concept of inner $ (\sigma,{\rm Id}) $-derivation recovers the concept of inner $ \sigma $-derivation if $ \mathbb{F} $ is a field, but this is not the case if $ \mathbb{F} $ is not commutative (see Definition \ref{def inner derivations} in Subsection \ref{subsec definition of phi} for the definition of inner derivations over general division rings).

The main result of this subsection is the following.

\begin{theorem} \label{theorem description vector der}
Let $ \sigma, \tau : \mathbb{F}_q \longrightarrow \mathbb{F}_q^{n \times n} $ be ring morphisms. The following hold:
\begin{enumerate}
\item
$ \sigma(a)\tau(b) = \tau(b) \sigma(a) $ for all $ a, b \in \mathbb{F}_q $ if, and only if, $ \sigma(c)\tau(c) = \tau(c) \sigma(c) $ for a primitive element $ c \in \mathbb{F}_q^* $.
\item
If Item 1 holds, then a map $ \delta : \mathbb{F}_q \longrightarrow \mathbb{F}_q^n $ is a $ (\sigma, \tau) $-derivation if, and only if, it is an inner $ (\sigma,\tau) $-derivation.
\end{enumerate}
In particular, for any ring morphism $ \sigma : \mathbb{F}_q \longrightarrow \mathbb{F}_q^{n \times n} $, the only $ (\sigma, \sigma) $-derivation is the zero derivation $ \delta = 0 $.
\end{theorem}

The first item in Theorem \ref{theorem description vector der} is straightforward to prove. We devote the rest of the section to proving the second item. 

The following auxiliary lemma is obtained by applying (\ref{eq multiplicative property vector derivations}) recursively on powers.

\begin{lemma} \label{lemma proof vector description 1}
Given a field $ \mathbb{F} $, ring morphisms $ \sigma, \tau : \mathbb{F} \longrightarrow \mathbb{F}^{n \times n} $ such that $ \sigma(a)\tau(a) = \tau(a) \sigma(a) $ for all $ a \in \mathbb{F} $, and a $ (\sigma, \tau) $-derivation $ \delta : \mathbb{F} \longrightarrow \mathbb{F}^n $, it holds that
\begin{equation}
\delta \left( a^{j+1} \right) = \left( \sum_{i=0}^j \sigma(a)^i \tau(a)^{j-i} \right) \delta(a),
\end{equation}
for all $ a \in \mathbb{F} $ and all integers $ j \geq 0 $.
\end{lemma}

Our second auxiliary lemma is the following fact about binomial coefficients over $ \mathbb{F}_q $. Since we were not able to find an explicit proof in the literature, we provide one for convenience.

\begin{lemma} \label{lemma proof vector description 2}
For every integer $ 0 \leq i \leq q-1 $, it holds that
$$ \binom{q-1}{i} = (-1)^i $$
inside the finite field $ \mathbb{F}_q $.
\end{lemma}
\begin{proof}
In the conventional polynomial ring $ \mathbb{F}_q[x,y] $ with commutative variables $ xy = yx $, it holds that
$$ (x-y)(x-y)^{q-1} = (x-y)^q = x^q - y^q = (x-y) \left( \sum_{i=0}^{q-1} x^iy^{q-1-i} \right). $$
Cancelling $ x-y \neq 0 $ on both sides and applying Newton's binomial formula for $ (x-y)^{q-1} $, we conclude that
$$ \sum_{i=0}^{q-1} (-1)^i \binom{q-1}{i} x^iy^{q-1-i} = (x-y)^{q-1} = \sum_{i=0}^{q-1} x^iy^{q-1-i}, $$
and the result follows.
\end{proof}

With these tools we may now complete the proof of Theorem \ref{theorem description vector der}.

\begin{proof}[Proof of Theorem \ref{theorem description vector der}]
Let $ c \in \mathbb{F}_q^* $ be a primitive element and define $ S = \sigma(c) \in \mathbb{F}_q^{n \times n} $ and $ T = \tau(c) \in \mathbb{F}_q^{n \times n} $. By Newton's binomial formula and Lemma \ref{lemma proof vector description 2}, it holds that 
\begin{equation}
(T - S)^{q-1} = \sum_{i=0}^{q-1} (-1)^i \binom{q-1}{i} T^i S^{q-1-i} = \sum_{i=0}^{q-1} T^i S^{q-1-i}.
\label{eq proof description vector der binomial formula}
\end{equation}
Now, by Equation (\ref{eq proof description vector der binomial formula}) and Lemma \ref{lemma proof vector description 1}, we have that 
\begin{equation}
\delta(c) = \delta(c^q) = \left( \sum_{i=0}^{q-1} T^i S^{q-1-i} \right) \delta(c) = (T-S)^{q-1} \delta(c).
\label{eq proof description vector der}
\end{equation}
Define now the vector (note that $ q \geq 2 $)
\begin{equation}
\boldsymbol\lambda = (T-S)^{q-2} \delta(c) \in \mathbb{F}_q^n . 
\label{eq def lambda for inne der}
\end{equation}
By definition of $ \boldsymbol\lambda $, $ S $ and $ T $, and by (\ref{eq proof description vector der}), we have that
$$ \delta(c) = (\tau(c) - \sigma(c)) \boldsymbol\lambda. $$  
Define then the map $ \partial : \mathbb{F}_q \longrightarrow \mathbb{F}_q^n $ by
$$ \partial(a) = (\tau(a) - \sigma(a)) \boldsymbol\lambda, $$
for all $ a \in \mathbb{F}_q $. Since $ \sigma(a) \tau(b) = \tau(b) \sigma(a) $ for all $ a,b \in \mathbb{F}_q $, by the first item in Theorem \ref{theorem description vector der}, we have that $ \partial $ is a $ (\sigma, \tau) $-derivation.

Finally, take any $ a \in \mathbb{F}_q^* $. There exists an integer $ 0 \leq j \leq q-1 $ such that $ a = c^{j+1} $. By Lemma \ref{lemma proof vector description 1} and the fact that $ \delta(c) = \partial(c) $, we conclude that
$$ \delta(a) = \delta \left( c^{j+1} \right) = \left( \sum_{i=0}^j \tau(c)^i \sigma(c)^{j-i} \right) \delta(c) = \left( \sum_{i=0}^j \tau(c)^i \sigma(c)^{j-i} \right) \partial(c) = \partial(a), $$
and we are done.
\end{proof}

\begin{remark}
Observe that (\ref{eq def lambda for inne der}) gives an explicit formula for $ \boldsymbol\lambda $ based on the evaluations of $ \sigma $, $ \tau $ and $ \delta $ at a primitive element of $ \mathbb{F}_q $.
\end{remark}

Note that over infinite fields there exist non-zero $ {\rm Id} $-derivations, which then cannot be inner, since any inner $ {\rm Id} $-derivation over a field is necessarily zero. This holds for any $ n \geq 1 $.

Recall also that, if $ n = 1 $, then all $ \sigma $-derivations are inner $ \sigma $-derivations even over infinite fields as long as $ \sigma \neq {\rm Id} $ (see \cite[Sec. 8.3]{cohn} or \cite[Prop. 39]{linearizedRS}). Interestingly, if $ n > 1 $, there are $ \sigma $-derivations that are not inner $ \sigma $-derivations over infinite fields even if $ \sigma \neq {\rm Id} $, as the next example shows.

\begin{example} \label{ex non inner derivation}
Let $ \mathbb{F} = \mathbb{F}_p (z) $ be the field of rational functions over $ \mathbb{F}_p $, for a prime number $ p $. Consider the ring morphism $ \sigma : \mathbb{F}_p (z) \longrightarrow \mathbb{F}_p (z)^{2 \times 2} $ given by
$$ \sigma(f(z)) = \left( \begin{array}{cc}
f(z) & f^\prime(z)  \\
0 & f(z)
\end{array} \right), $$
for $ f(z) \in \mathbb{F}_p (z) $, where $ f(z) \mapsto f^\prime(z) $ is the usual derivation in $ \mathbb{F}_p (z) $. Then the map $ \delta : \mathbb{F}_p (z) \longrightarrow \mathbb{F}_p (z)^2 $, given by
$$ \delta(f(z)) = \left( \begin{array}{c}
f^{\prime \prime}(z) \\
2 f^\prime(z) 
\end{array} \right), $$
is a $ \sigma $-derivation that is not an inner $ \sigma $-derivation. The proof is left to the reader.
\end{example}

\subsection{All morphisms are diagonalizable} \label{subsec morphisms}

In this subsection, we will prove that we may ``diagonalize'' all ring morphisms $ \sigma: \mathbb{F}_q \longrightarrow \mathbb{F}_q^{n \times n} $. The formal definition is as follows.

\begin{definition} \label{def diagonalizable morphisms}
Given an arbitrary division ring $ \mathbb{F} $ and a ring morphism $ \sigma : \mathbb{F} \longrightarrow \mathbb{F}^{n \times n} $, we say that $ \sigma $ is diagonalizable if there exists an invertible matrix $ A \in \mathbb{F}^{n \times n} $ and ring endomorphisms $ \sigma_1, \sigma_2, \ldots, \sigma_n : \mathbb{F} \longrightarrow \mathbb{F} $ such that
$$ \sigma(a) = A {\rm diag}(\sigma_1(a), \sigma_2(a), \ldots, \sigma_n(a)) A^{-1} = A \left( \begin{array}{cccc}
\sigma_1 (a) & 0 & \ldots & 0 \\
0 & \sigma_2 (a) & \ldots & 0 \\
\vdots & \vdots & \ddots & \vdots \\
0 & 0 & \ldots & \sigma_n (a) \\
\end{array} \right) A^{-1}, $$
for all $ a \in \mathbb{F} $.
\end{definition}

Note that if $ \tau : \mathbb{F} \longrightarrow \mathbb{F}^{n \times n} $ is a ring morphism over a division ring $ \mathbb{F} $, then so is $ \sigma : \mathbb{F} \longrightarrow \mathbb{F}^{n \times n} $ given by $ \sigma(a) = A \tau(a) A^{-1} $, for all $ a \in \mathbb{F} $, and for an invertible matrix $ A \in \mathbb{F}^{n \times n} $. See also Proposition \ref{prop conjugation of morphism and derivation} in Subsection \ref{subsec def of varphi}.

The main result of this section is the following.

\begin{theorem} \label{theorem description matrix morph}
Any ring morphism $ \sigma : \mathbb{F}_q \longrightarrow \mathbb{F}_q^{n \times n} $ over a finite field $ \mathbb{F}_q $ is diagonalizable.
\end{theorem}
\begin{proof}
Fix a primitive element $ c \in \mathbb{F}_q^* $, and define $ S = \sigma(c) \in \mathbb{F}_q^{n \times n} $ and $ F(t) = \det(S - tI) \in \mathbb{F}_q[t] $. Since any finite algebraic extension of a finite field is again a finite field, we have that $ F(t) = \prod_{i=1}^k (t - \lambda_i)^{n_i} $, where $ \lambda_1, \lambda_2, \ldots, \lambda_k \in \mathbb{F}_{q^r} $ are pair-wise distinct, for a suitable finite field extension $ \mathbb{F}_{q^r} $ of $ \mathbb{F}_q $ and for positive integers $ n_1, n_2, \ldots, n_k $ such that $ n = n_1 + n_2 + \cdots + n_k $. 

Hence $ S $ admits a Jordan canonical form in $ \mathbb{F}_{q^r}^{n \times n} $ (see for instance \cite[Th. 25, p. 354]{birkhoff} or \cite[Sec. V.9]{lang-under}). In other words, there exists an invertible matrix $ A \in \mathbb{F}_{q^r}^{n \times n} $ such that
\begin{equation}
S = A \left( \begin{array}{cccc}
\Lambda_1 & 0 & \ldots & 0 \\
0 & \Lambda_2 & \ldots & 0 \\
\vdots & \vdots & \ddots & \vdots \\
0 & 0 & \ldots & \Lambda_n \\
\end{array} \right) A^{-1},
\label{eq jordan form primitive}
\end{equation}
where the matrix $ \Lambda_i $ is given by
\begin{equation}
\Lambda_i = \left( \begin{array}{cccccc}
\lambda_i & \varepsilon_{i,1} & 0 & \ldots & 0 & 0 \\
0 & \lambda_i & \varepsilon_{i,2} & \ldots & 0 & 0 \\
0 & 0 & \lambda_i & \ldots & 0 & 0 \\
\vdots & \vdots & \vdots & \ddots & \vdots & \vdots \\
0 & 0 & 0 & \ldots & \lambda_i & \varepsilon_{i,n_i-1} \\
0 & 0 & 0 & \ldots & 0 & \lambda_i \\
\end{array} \right) \in \mathbb{F}_{q^r}^{n_i \times n_i} ,
\label{eq each submatrix}
\end{equation}
where $ \varepsilon_{i,j} $ is either $ 1 $ or $ 0 $, for $ j = 1,2, \ldots, n_i-1 $, and for $ i = 1,2, \ldots, k $.

Now, for an arbitrary $ a \in \mathbb{F}_q^* $, there exists an integer $ 0 \leq j \leq q-1 $ such that $ a = c^j $. Using the fact that $ \sigma(a) = \sigma(c)^j $ and the expressions (\ref{eq jordan form primitive}) and (\ref{eq each submatrix}), we see that there exist maps $ \tau_i : \mathbb{F}_q \longrightarrow \mathbb{F}_{q^r}^{n_i \times n_i} $, for $ i = 1,2, \ldots, k $, such that
$$ \sigma(a) = A \left( \begin{array}{cccc}
\tau_1 (a) & 0 & \ldots & 0 \\
0 & \tau_2 (a) & \ldots & 0 \\
\vdots & \vdots & \ddots & \vdots \\
0 & 0 & \ldots & \tau_k (a) \\
\end{array} \right) A^{-1}, $$
for all $ a \in \mathbb{F}_q $. Moreover, there exist functions $ \varphi^{(i)} : \mathbb{F}_q \longrightarrow \mathbb{F}_{q^r} $ and $ \varphi_{u,v}^{(i)} : \mathbb{F}_q \longrightarrow \mathbb{F}_{q^r} $, for $ 1 \leq u < v \leq n_i $, such that
$$ \tau_i(a) = \left( \begin{array}{cccc}
\varphi^{(i)} (a) & \varphi_{1,2}^{(i)} (a) & \ldots & \varphi_{1,n_i}^{(i)} (a) \\
0 & \varphi^{(i)} (a) & \ldots & \varphi_{2,n_i}^{(i)} (a) \\
\vdots & \vdots & \ddots & \vdots \\
0 & 0 & \ldots & \varphi^{(i)} (a) \\
\end{array} \right), $$
for all $ a \in \mathbb{F}_q $, and for $ i = 1,2, \ldots, k $. 

Note that $ \tau : \mathbb{F}_q \longrightarrow \mathbb{F}_{q^r}^{n \times n} $, given by $ \tau(a) = A^{-1}\sigma(a) A $, for $ a \in \mathbb{F}_q $, is a ring morphism, as observed in Proposition \ref{prop conjugation of morphism and derivation} in Subsection \ref{subsec def of varphi}. In particular, the functions $ \varphi^{(i)} : \mathbb{F}_q \longrightarrow \mathbb{F}_{q^r} $ are field morphisms, for $ i = 1,2, \ldots, k $. Thus $ \varphi^{(i)}(a)^q = \varphi^{(i)}(a^q) = \varphi^{(i)}(a) $ for all $ a \in \mathbb{F}_q $, hence $ \varphi^{(i)} : \mathbb{F}_q \longrightarrow \mathbb{F}_q $, and we deduce that $ \lambda_i = \varphi^{(i)}(c) \in \mathbb{F}_q $ and, therefore, $ \varphi_{u,v}^{(i)} : \mathbb{F}_q \longrightarrow \mathbb{F}_q $, for all $ 1 \leq u < v \leq n_i $ and for all $ i = 1,2, \ldots, k $. In other words, $ \tau_i : \mathbb{F}_q \longrightarrow \mathbb{F}_q^{n_i \times n_i} $, for all $ i = 1,2, \ldots, k $, and $ \tau : \mathbb{F}_q \longrightarrow \mathbb{F}_q^{n \times n} $. In particular, $ \sigma(c), \tau(c) \in \mathbb{F}_q^{n \times n} $ are similar over $ \mathbb{F}_{q^r} $ (i.e., $ \sigma(c) = A \tau(c) A^{-1} $ and $ A \in \mathbb{F}_{q^r}^{n \times n} $ is invertible), hence they are also similar over $ \mathbb{F}_q $ (since they share the same rational canonical form, see \cite[Th. 42, p. 353]{birkhoff}), and therefore we may assume that $ A \in \mathbb{F}_q^{n \times n} $.

Next, the fact that $ \tau_i : \mathbb{F}_q \longrightarrow \mathbb{F}_q^{n_i \times n_i} $ is a ring morphism implies that $ \varphi_{j,j+1}^{(i)} : \mathbb{F}_q \longrightarrow \mathbb{F}_q $ are $ (\varphi^{(i)}, \varphi^{(i)}) $-derivations, for $ j = 1,2, \ldots, n_i-1 $, and for $ i = 1,2, \ldots, k $. By Theorem \ref{theorem description vector der}, it follows that $ \varphi_{j,j+1}^{(i)} (a) = 0 $ for all $ a \in \mathbb{F}_q $, and in particular, $ \varepsilon_{i,j} = \varphi_{j,j+1}^{(i)} (c) = 0 $, for all $ j = 1,2, \ldots, n_i-1 $ and all $ i = 1,2, \ldots, k $. Hence, we conclude that
$$ \tau_i(a) = \left( \begin{array}{cccc}
\varphi^{(i)} (a) & 0 & \ldots & 0 \\
0 & \varphi^{(i)} (a) & \ldots & 0 \\
\vdots & \vdots & \ddots & \vdots \\
0 & 0 & \ldots & \varphi^{(i)} (a) \\
\end{array} \right), $$
for all $ a \in \mathbb{F}_q $, and the result follows.
\end{proof}

\begin{remark}
Observe that the proof above gives an algorithmic method to diagonalize $ \sigma $, since it is based on diagonalizing the matrix $ \sigma(c) \in \mathbb{F}_q^{n \times n} $ for a primitive element $ c \in \mathbb{F}_q^* $. 
\end{remark}

\begin{remark}
Recall that, although any square matrix over a finite field admits a Jordan canonical form over its algebraic closure (thus over a finite extension), when it comes to diagionalization, any scenario may happen. As simple examples, the matrix
$$ \left( \begin{array}{ccc}
0 & 0 & -1 \\
1 & 0 & 1 \\
0 & 1 & 0
\end{array} \right) \in \mathbb{F}_3^{3 \times 3} $$
is diagonalizable over $ \mathbb{F}_{27} $ but not over $ \mathbb{F}_3 $, and the matrix
$$ \left( \begin{array}{cc}
1 & 1 \\
1 & 1
\end{array} \right) \in \mathbb{F}_2^{2 \times 2} $$
is not diagonalizable over any field extension of $ \mathbb{F}_2 $, since it is non-zero and nilpotent.
\end{remark}

Note that the well-known Skolem-Noether theorem \cite{skolem} cannot be directly applied to prove Theorem \ref{theorem description matrix morph}. Following \cite[Th. 2.10]{milne}, the theorem in general form reads as follows.

\begin{theorem} [\textbf{Skolem-Noether}] \label{th skolem-noether}
Let $ \mathbb{F} $ be a field. Given simple (and possibly non-commutative) $ \mathbb{F} $-algebras $ \mathcal{A} $ and $ \mathcal{B} $ such that $ \mathbb{F} $ is the center of $ \mathcal{B} $, then for any two $ \mathbb{F} $-algebra morphisms $ \sigma, \tau : \mathcal{A} \longrightarrow \mathcal{B} $, there exists an invertible element $ A \in \mathcal{B} $ such that 
\begin{equation}
\sigma(a) = A \tau(a) A^{-1},
\label{eq similar morphisms for skolem-noether}
\end{equation}
for all $ a \in \mathcal{A} $. 
\end{theorem}

Observe that the assumptions in Theorem \ref{th skolem-noether} do not hold in the case of Theorem \ref{theorem description matrix morph} when setting $ \mathcal{A} = \mathbb{F}_q $ and $ \mathcal{B} = \mathbb{F}_q^{n \times n} $ (resp. $ \mathcal{B} = \sigma(\mathbb{F}_q) \subseteq \mathbb{F}_q^{n \times n} $). First, the center of $ \mathbb{F}_q^{n \times n} $ contains $ \mathbb{F}_q $ (resp. $ \sigma(\mathbb{F}_q) $). However, the map $ \sigma $ in Theorem \ref{theorem description matrix morph} is not linear over $ \mathbb{F}_q $ (resp. $ \sigma(\mathbb{F}_q) \cong \mathbb{F}_q $) in general. Hence the assumptions in Theorem \ref{th skolem-noether} do not hold. 

In fact, it is not true that any two ring morphisms $ \sigma, \tau : \mathbb{F}_q \longrightarrow \mathbb{F}_q^{n \times n} $ are related as in (\ref{eq similar morphisms for skolem-noether}), even if $ \sigma(\mathbb{F}_q) = A \tau(\mathbb{F}_q) A^{-1} $ for all invertible matrices $ A \in \mathbb{F}_q^{n \times n} $. We show this in the next example. Note that this observation is trivial when $ n = 1 $.

\begin{example} \label{ex two non similar morphisms}
We will fix $ n = 2 $ for illustration purposes, although this example works for any $ n \geq 1 $ and it is trivial for $ n = 1 $. Let $ q = p^4 $ for some prime number $ p $, and let $ \sigma : \mathbb{F}_q \longrightarrow \mathbb{F}_q^{2 \times 2} $ be given by
$$ \sigma(a) = \left( \begin{array}{cc}
a^{p^2} & 0 \\
0 & a^{p^2}
\end{array} \right), $$
for all $ a \in \mathbb{F}_q $. Since $ (a^{p^2})^{p^2} = a^{p^4} = a $, for all $ a \in \mathbb{F}_q $, we deduce that $ \sigma(\mathbb{F}_q) = \{ a I \in \mathbb{F}_q^{2 \times 2} \mid a \in \mathbb{F}_q \} $. Let $ \tau = {\rm Id} : \mathbb{F}_q \longrightarrow \mathbb{F}_q^{2 \times 2} $. Obviously, it holds that $ \sigma(\mathbb{F}_q) = \tau(\mathbb{F}_q) $, or even further, $ \sigma(\mathbb{F}_q) = A \tau(\mathbb{F}_q) A^{-1} $, for all invertible matrices $ A \in \mathbb{F}_q^{2 \times 2} $. However, if there existed an invertible matrix $ A \in \mathbb{F}_q^{2 \times 2} $ such that $ \sigma(a) = A \tau(a) A^{-1} $, for all $ a \in \mathbb{F}_q $, then the elements fixed by $ \sigma $ (i.e., those $ a \in \mathbb{F}_q $ such that $ \sigma(a) = aI $) would be the same as those fixed by $ \tau $. This is not true, as the elements fixed by $ \sigma $ form $ \mathbb{F}_{p^2} $, whereas those fixed by $ \tau $ form $ \mathbb{F}_{p^4} $.
\end{example}

Example \ref{ex two non similar morphisms} also shows that we lose information on similarity when looking at the set $ \sigma(\mathbb{F}_q) \subseteq \mathbb{F}_q^{n \times n} $. Observe that, if $  \mathbb{F} $ is a field and $ \sigma : \mathbb{F} \longrightarrow \mathbb{F}^{n \times n} $ is a ring morphism, then $ \sigma(\mathbb{F}) $ is a subfield of the ring $ \mathbb{F}^{n \times n} $. Subfields of $ \mathbb{F}^{n \times n} $ are generally called matrix fields. Two subfields $ \mathcal{A}, \mathcal{B} \subseteq \mathbb{F}^{n \times n} $ are called similar if there exists an invertible matrix $ A \in \mathbb{F}^{n \times n} $ such that $ \mathcal{A} = A \mathcal{B} A^{-1} $. Matrix fields over finite fields and their similarity have been studied earlier, for instance in \cite{beard2, beard1}. However, as Example \ref{ex two non similar morphisms} shows, it may hold that $ \sigma(\mathbb{F}_q), \tau(\mathbb{F}_q) \subseteq \mathbb{F}_q^{n \times n} $ are similar matrix fields (or even equal), but still $ \sigma, \tau : \mathbb{F}_q \longrightarrow \mathbb{F}_q^{n \times n} $ are not similar ring morphisms (i.e., there is no invertible matrix $ A \in \mathbb{F}_q^{n \times n} $ such that $ \sigma(a) = A \tau(a) A^{-1} $, for all $ a \in \mathbb{F}_q $).

We conclude by showing that such a diagonalization of ring morphisms $ \sigma : \mathbb{F} \longrightarrow \mathbb{F}^{n \times n} $ is not always possible over infinite fields.

\begin{example} \label{ex non diagonalizable morphism}
Let $ \mathbb{F} = \mathbb{F}_p (z) $ be the field of rational functions over $ \mathbb{F}_p $, for a prime number $ p $. The ring morphism $ \sigma : \mathbb{F}_p (z) \longrightarrow \mathbb{F}_p (z)^{2 \times 2} $ given by
$$ \sigma(f(z)) = \left( \begin{array}{cc}
f(z) & f^\prime(z) \\
0 & f(z)
\end{array} \right), $$
for $ f(z) \in \mathbb{F}_p (z) $, as in Example \ref{ex non inner derivation}, is not diagonalizable: Similar to the previous example, the subfield of $ \mathbb{F}_p (z) $ fixed by $ \sigma $ is $ \mathbb{F}_p (z^p) $, but there is no field endomorphism of $ \mathbb{F}_p (z) $, other than the identity, leaving the elements in $ \mathbb{F}_p (z^p) $ fixed (see also \cite[Ex. 48]{linearizedRS}).
\end{example}

\section{Linear transformations of variables} \label{sec linear transformations}

In the previous section, we showed that all ring morphisms $ \sigma : \mathbb{F}_q \longrightarrow \mathbb{F}_q^{n \times n} $ over a finite field $ \mathbb{F}_q $ are diagonalizable (Theorem \ref{theorem description matrix morph}), that is, they are similar (i.e., as in Proposition \ref{prop conjugation of morphism and derivation} below) to a diagonal ring morphism. 

In this section, we show that similar ring morphisms allow to naturally define linear transformations of variables between the corresponding multivariate skew polynomial rings. Furthermore, such transformations are ring isomorphisms and preserve evaluations and degrees, hence can be naturally extended to non-free multivariate skew polynomial rings (see Section \ref{sec preliminaries}). Combined with Theorem \ref{theorem description matrix morph}, these linear transformations will show in Section \ref{sec classification finite fields} how to simplify and classify multivariate skew polynomial rings over finite fields.

Each subsection is devoted to a different property of the mentioned ring isomorphism, which will be denoted by $ \varphi_A $ and depends on the invertible matrix $ A \in \mathbb{F}^{n \times n} $ that gives the linear transformation. Throughout this section, $ \mathbb{F} $ may be any division ring.

\subsection{Definition of the map $ \varphi_A $} \label{subsec def of varphi}

We start by formalizing the following observation, which we have mentioned earlier in the paper.

\begin{proposition} \label{prop conjugation of morphism and derivation}
Let $ \tau : \mathbb{F} \longrightarrow \mathbb{F}^{n \times n} $ be a ring morphism and let $ \delta_\tau : \mathbb{F} \longrightarrow \mathbb{F}^n $ be a $ \tau $-derivation. For an invertible matrix $ A \in \mathbb{F}^{n \times n} $, define the maps $ \sigma : \mathbb{F} \longrightarrow \mathbb{F}^{n \times n} $ and $ \delta_\sigma : \mathbb{F} \longrightarrow \mathbb{F}^n $ by
$$ \sigma (a) = A \tau (a) A^{-1} \quad \textrm{and} \quad \delta_\sigma(a) = A \delta_\tau(a), $$
for all $ a \in \mathbb{F} $. Then $ \sigma $ a ring morphism and $ \delta $ is a $ \sigma $-derivation.
\end{proposition}

Linear transformations of variables are then defined as follows.

\begin{definition} \label{def map phi A}
Let $ \sigma, \tau : \mathbb{F} \longrightarrow \mathbb{F}^{n \times n} $ be ring morphisms and let $ \delta_\sigma, \delta_\tau : \mathbb{F} \longrightarrow \mathbb{F}^n $ be a $ \sigma $-derivation and a $ \tau $-derivation, respectively, as in the previous proposition, for an invertible matrix $ A \in \mathbb{F}^{n \times n} $. We define the map
\begin{equation}
\varphi_A : \mathbb{F}[\mathbf{x}; \sigma, \delta_\sigma] \longrightarrow \mathbb{F}[\mathbf{x}; \tau, \delta_\tau]
\label{eq definition of phi A}
\end{equation}
as follows. First, we define $ \varphi_A(1) = 1 $ and
$$ \varphi_A(\mathbf{x}) = \left( \begin{array}{c}
\varphi_A(x_1) \\
\varphi_A(x_2) \\
\vdots \\
\varphi_A(x_n)
\end{array} \right) = \left( \begin{array}{c}
\sum_{j=1}^n a_{1,j}x_j \\
\sum_{j=1}^n a_{2,j}x_j \\
\vdots \\
\sum_{j=1}^n a_{n,j}x_j
\end{array} \right) = A \mathbf{x}. $$
Next we define $ \varphi_A(\mathfrak{m}(\mathbf{x})) $ recursively on monomials $ \mathfrak{m}(\mathbf{x}) \in \mathcal{M} $. Assume that $ \varphi_A(\mathfrak{m}(\mathbf{x})) $ is defined, for a given $ \mathfrak{m}(\mathbf{x}) \in \mathcal{M} $. Then we define
$$ \varphi_A(\mathbf{x} \mathfrak{m}(\mathbf{x})) = \left( \begin{array}{c}
\varphi_A(x_1 \mathfrak{m}(\mathbf{x})) \\
\varphi_A(x_2 \mathfrak{m}(\mathbf{x})) \\
\vdots \\
\varphi_A(x_n \mathfrak{m}(\mathbf{x}))
\end{array} \right) = \left( \begin{array}{c}
\sum_{j=1}^n a_{1,j}x_j \varphi_A(\mathfrak{m}(\mathbf{x})) \\
\sum_{j=1}^n a_{2,j}x_j \varphi_A(\mathfrak{m}(\mathbf{x})) \\
\vdots \\
\sum_{j=1}^n a_{n,j}x_j \varphi_A(\mathfrak{m}(\mathbf{x}))
\end{array} \right) = A \mathbf{x} \varphi_A(\mathfrak{m}(\mathbf{x})). $$
Finally, if $ F(\mathbf{x}) = \sum_{\mathfrak{m}(\mathbf{x}) \in \mathcal{M}} F_\mathfrak{m} \mathfrak{m}(\mathbf{x}) \in \mathbb{F}[\mathbf{x}; \sigma, \delta_\sigma] $, where $ F_\mathfrak{m} \in \mathbb{F} $, for all $ \mathfrak{m}(\mathbf{x}) \in \mathcal{M} $, we define
\begin{equation}
\varphi_A(F(\mathbf{x})) = \sum_{\mathfrak{m}(\mathbf{x}) \in \mathcal{M}} F_\mathfrak{m} \varphi_A(\mathfrak{m}(\mathbf{x})).
\label{eq def by linearity of phi A}
\end{equation}
Given $ F(\mathbf{x}) \in \mathbb{F}[\mathbf{x}; \sigma, \delta_\sigma] $ or $ \mathfrak{m}(\mathbf{x}) \in \mathcal{M} $, we will use the notation
\begin{equation}
F(A \mathbf{x}) = \varphi_A(F(\mathbf{x})) \quad \textrm{and} \quad \mathfrak{m}(A \mathbf{x}) = \varphi_A(\mathfrak{m} (\mathbf{x})).
\label{eq notation linear transform}
\end{equation}
\end{definition} 

\begin{remark}
Observe that, if $ \mathbb{F} $ is a field and $ \tau = {\rm Id} $, then $ \sigma = \tau $ and $ \delta_\sigma = A \delta_\tau $ represents a linear transformation of the standard derivations in $ \delta_\tau $. In case $ \delta_\tau = 0 $, then $ \delta_\sigma = 0 $, and both skew polynomial rings in (\ref{eq definition of phi A}) are the conventional free multivariate polynomial ring $ \mathbb{F}[\mathbf{x}] $. In that case, $ \varphi_A(F(\mathbf{x})) $ coincides with the usual definition of $ F(A \mathbf{x}) $ that consists in substituting $ \mathbf{x} $ by $ A \mathbf{x} $, for any $ F(\mathbf{x}) \in \mathbb{F}[\mathbf{x}] $. Observe that if $ \mathbb{F} $ is a non-commutative division ring and $ \tau = {\rm Id} $, then the ring morphism $ \sigma : \mathbb{F} \longrightarrow \mathbb{F}^{n \times n} $ given by $ \sigma(a) = A (aI) A^{-1} $, for $ a \in \mathbb{F} $, is not necessarily the identity if the entries of $ A $ do not lie in the center of $ \mathbb{F} $.
\end{remark}

\subsection{The map $ \varphi_A $ is additive and multiplicative}

In this subsection, we will prove that the map $ \varphi_A $ in Definition \ref{def map phi A} is a ring morphism.

\begin{proposition} \label{prop phi A is ring morphism}
Let $ \sigma, \tau : \mathbb{F} \longrightarrow \mathbb{F}^{n \times n} $ be ring morphisms and let $ \delta_\sigma, \delta_\tau : \mathbb{F} \longrightarrow \mathbb{F}^n $ be a $ \sigma $-derivation and a $ \tau $-derivation, respectively, as in Proposition \ref{prop conjugation of morphism and derivation}, for an invertible matrix $ A \in \mathbb{F}^{n \times n} $. The map $ \varphi_A : \mathbb{F}[\mathbf{x}; \sigma, \delta_\sigma] \longrightarrow \mathbb{F}[\mathbf{x}; \tau, \delta_\tau] $ in Definition \ref{def map phi A} is a ring morphism. That is, with notation as in (\ref{eq notation linear transform}), it holds that
$$ (F + G)(A \mathbf{x}) = F(A\mathbf{x}) + G(A\mathbf{x}) \quad \textrm{and} \quad (FG)(A \mathbf{x}) = F(A\mathbf{x}) G(A\mathbf{x}), $$
for all $ F(\mathbf{x}), G(\mathbf{x}) \in \mathbb{F}[\mathbf{x}; \sigma, \delta_\sigma] $, where the operations on the right-hand sides are in $ \mathbb{F}[\mathbf{x}; \tau, \delta_\tau] $, and the operations on the left-hand sides are in $ \mathbb{F}[\mathbf{x}; \sigma, \delta_\sigma] $.
\end{proposition}
\begin{proof}
The fact that $ \varphi_A $ is additive (or even left linear over $ \mathbb{F} $) follows directly from the definitions.

Now, due to (\ref{eq def by linearity of phi A}), in order to prove that $ \varphi_A $ is multiplicative, we only need to show that
$$ \varphi_A(\mathfrak{m}(\mathbf{x}) G(\mathbf{x})) = \varphi_A(\mathfrak{m}(\mathbf{x})) \varphi_A(G(\mathbf{x})), $$
for all $ \mathfrak{m}(\mathbf{x}) \in \mathcal{M} $ and all $ G(\mathbf{x}) \in \mathbb{F}[\mathbf{x}; \sigma, \delta_\sigma] $. We do this recursively on $ \mathfrak{m}(\mathbf{x}) \in \mathcal{M} $.

First observe that it is trivial for $ \mathfrak{m}(\mathbf{x}) = 1 $. To perform the induction step, we need to prove the result for $ \mathfrak{m}(\mathbf{x}) = x_i $, for $ i = 1,2, \ldots, n $. Denote $ G(\mathbf{x}) = \sum_{\mathfrak{m}(\mathbf{x}) \in \mathcal{M}} G_\mathfrak{m} \mathfrak{m}(\mathbf{x}) $, where $ G_\mathfrak{m} \in \mathbb{F} $, for all $ \mathfrak{m}(\mathbf{x}) \in \mathcal{M} $. First it holds that
$$ \mathbf{x} G(\mathbf{x}) = \sum_{\mathfrak{m}(\mathbf{x}) \in \mathcal{M}} \mathbf{x}(G_\mathfrak{m} \mathfrak{m}(\mathbf{x})) = \sum_{\mathfrak{m}(\mathbf{x}) \in \mathcal{M}} (\sigma(G_\mathfrak{m}) \mathbf{x} \mathfrak{m}(\mathbf{x}) + \delta_\sigma(G_\mathfrak{m}) \mathfrak{m}(\mathbf{x}) ), $$
since these operations are in the ring $ \mathbb{F}[\mathbf{x}; \sigma, \delta_\sigma] $. Therefore, by definition, it holds that
$$ \varphi_A( \mathbf{x} G(\mathbf{x})) = \sum_{\mathfrak{m}(\mathbf{x}) \in \mathcal{M}} (\sigma(G_\mathfrak{m}) (A \mathbf{x} \mathfrak{m}(A\mathbf{x})) + \delta_\sigma(G_\mathfrak{m}) \mathfrak{m}(A\mathbf{x}) ). $$
Next, working in the ring $ \mathbb{F}[\mathbf{x}; \tau, \delta_\tau] $, we see that
$$ \varphi_A( \mathbf{x}) \varphi_A(G(\mathbf{x})) = (A \mathbf{x}) G(A \mathbf{x}) = A \sum_{\mathfrak{m}(\mathbf{x}) \in \mathcal{M}} (\tau(G_\mathfrak{m}) \mathbf{x} \mathfrak{m}(A\mathbf{x}) + \delta_\tau(G_\mathfrak{m}) \mathfrak{m}(A\mathbf{x}) ) = $$
$$ \sum_{\mathfrak{m}(\mathbf{x}) \in \mathcal{M}} (A\tau(G_\mathfrak{m})A^{-1} (A \mathbf{x} \mathfrak{m}(A\mathbf{x})) + A\delta_\tau(G_\mathfrak{m}) \mathfrak{m}(A\mathbf{x}) ) = $$
$$ \sum_{\mathfrak{m}(\mathbf{x}) \in \mathcal{M}} (\sigma(G_\mathfrak{m}) (A \mathbf{x} \mathfrak{m}(A\mathbf{x})) + \delta_\sigma(G_\mathfrak{m}) \mathfrak{m}(A\mathbf{x}) ). $$
Thus we deduce that
\begin{equation}
\varphi_A( \mathbf{x} G(\mathbf{x})) = \varphi_A( \mathbf{x}) \varphi_A(G(\mathbf{x})),
\label{eq proof multiplicative for x}
\end{equation}
where operations on the right-hand side are in the ring $ \mathbb{F}[\mathbf{x}; \tau, \delta_\tau] $.

Finally, assume that the result holds for a given $ \mathfrak{m}(\mathbf{x}) \in \mathcal{M} $, and we will prove it for $ x_i \mathfrak{m}(\mathbf{x}) $, for $ i = 1,2, \ldots, n $. It follows from (\ref{eq proof multiplicative for x}) and the hypothesis on $ \mathfrak{m}(\mathbf{x}) $ that
$$ \varphi_A(\mathbf{x} (\mathfrak{m}(\mathbf{x}) G(\mathbf{x})) ) = (A \mathbf{x}) (\mathfrak{m}(A\mathbf{x}) G(A \mathbf{x})) = $$
$$ ((A \mathbf{x}) \mathfrak{m}(A\mathbf{x})) G(A\mathbf{x}) = \varphi_A(\mathbf{x} (\mathfrak{m}(\mathbf{x})) \varphi_A(G(\mathbf{x})) ), $$
and we are done.
\end{proof}

\subsection{The map $ \varphi_A $ preserves evaluations}

In this subsection, we show that the $ (\tau, \delta_\tau) $-evaluation of $ F(A\mathbf{x}) \in \mathbb{F}[\mathbf{x}; \tau, \delta_\tau] $ at an affine point $ \mathbf{a} \in \mathbb{F}^n $ coincides with the $ (\sigma,\delta_\sigma) $-evaluation of $ F(\mathbf{x}) \in \mathbb{F}[\mathbf{x}; \sigma, \delta_\sigma] $ at the affine point $ A \mathbf{a} \in \mathbb{F}^n $. See (\ref{eq evaluation}) in Section \ref{sec preliminaries} for the definition of evaluation.

\begin{proposition} \label{prop varphi preserves evaluations}
Let $ \sigma, \tau : \mathbb{F} \longrightarrow \mathbb{F}^{n \times n} $ be ring morphisms and let $ \delta_\sigma, \delta_\tau : \mathbb{F} \longrightarrow \mathbb{F}^n $ be a $ \sigma $-derivation and a $ \tau $-derivation, respectively, as in Proposition \ref{prop conjugation of morphism and derivation}, for an invertible matrix $ A \in \mathbb{F}^{n \times n} $. The map $ \varphi_A : \mathbb{F}[\mathbf{x}; \sigma, \delta_\sigma] \longrightarrow \mathbb{F}[\mathbf{x}; \tau, \delta_\tau] $ in Definition \ref{def map phi A} preserves evaluations over any point $ \mathbf{a} \in \mathbb{F}^n $ after multiplication with the matrix $ A $. More concretely, according to the definition of evaluation in (\ref{eq evaluation}), it holds that
\begin{equation}
E_{\mathbf{a}}^{\tau, \delta_\tau}(F(A \mathbf{x})) = E_{A \mathbf{a}}^{\sigma, \delta_\sigma}(F(\mathbf{x})),
\label{eq varphi preserves evalautions}
\end{equation}
for all $ F(\mathbf{x}) \in \mathbb{F}[\mathbf{x}; \sigma, \delta_\sigma] $ and all $ \mathbf{a} \in \mathbb{F}^n $. In other words, the evaluation of $ F(A \mathbf{x}) $ in the point $ \mathbf{a} $ is $ F(A \mathbf{a}) $.
\end{proposition}
\begin{proof}
Fix $ F(\mathbf{x}) \in \mathbb{F}[\mathbf{x}; \sigma, \delta_\sigma] $ and $ \mathbf{a} \in \mathbb{F}^n $, and denote $ \mathbf{b} = (b_1, b_2, \ldots, b_n) = A \mathbf{a} \in \mathbb{F}^n $. By the results in Section \ref{sec preliminaries}, there exist $ G_1(\mathbf{x}), G_2(\mathbf{x}), \ldots, $ $ G_n(\mathbf{x}) \in \mathbb{F}[\mathbf{x}; \sigma, \delta_\sigma] $ such that
$$ F(\mathbf{x}) = \sum_{i = 1}^n G_i(\mathbf{x}) (x_i - b_i) + F(A \mathbf{a}) = \mathbf{G}(\mathbf{x})^T \cdot (\mathbf{x} - \mathbf{b}) + F(A \mathbf{a}), $$
where $ \mathbf{G}(\mathbf{x}) \in \mathbb{F}[\mathbf{x}; \sigma, \delta_\sigma]^n $ is the column vector of length $ n $ whose $ i $th row is $ G_i(\mathbf{x}) $, for $ i = 1,2, \ldots, n $. Now, using that $ \varphi_A $ is a ring morphism (Proposition \ref{prop phi A is ring morphism}), we deduce that
$$ F(A \mathbf{x}) = \mathbf{G}(A \mathbf{x})^T \cdot (A \mathbf{x} - A \mathbf{a}) + F(A \mathbf{a}) = (\mathbf{G}(A \mathbf{x})^T A) \cdot (\mathbf{x} - \mathbf{a}) + F(A \mathbf{a}), $$
and the result follows by the uniqueness of the remainder $ F(A \mathbf{a}) \in \mathbb{F} $ (see Section \ref{sec preliminaries}).
\end{proof}

Thanks to this result, we may extend $ \varphi_A $ to non-free multivariate skew polynomial rings. We follow the definitions in Section \ref{sec preliminaries}. However, for clarity, we denote by $ I^{\tau, \delta_\tau}(\mathbb{F}^n) $ and $ I^{\sigma, \delta_\sigma}(\mathbb{F}^n) $ the two-sided ideals of $ \mathbb{F}[\mathbf{x}; \tau, \delta_\tau] $ and $ \mathbb{F}[\mathbf{x}; \sigma, \delta_\sigma] $, respectively, given by skew polynomials that vanish at every point. We have the following consequence of Proposition \ref{prop varphi preserves evaluations}.

\begin{corollary} \label{cor extend varphi to non-free}
Let $ \sigma, \tau : \mathbb{F} \longrightarrow \mathbb{F}^{n \times n} $ be ring morphisms and let $ \delta_\sigma, \delta_\tau : \mathbb{F} \longrightarrow \mathbb{F}^n $ be a $ \sigma $-derivation and a $ \tau $-derivation, respectively, as in Proposition \ref{prop conjugation of morphism and derivation}, for an invertible matrix $ A \in \mathbb{F}^{n \times n} $. It holds that
$$ \varphi_A \left( I^{\sigma, \delta_\sigma}(\mathbb{F}^n) \right) = I^{\tau, \delta_\tau}(\mathbb{F}^n). $$
Therefore, we may extend $ \varphi_A $ to a ring morphism between non-free multivariate skew polynomial rings, that is, to ring morphisms
\begin{equation}
\varphi_A : \mathbb{F}[\mathbf{x}; \sigma, \delta_\sigma]/I \longrightarrow \mathbb{F}[\mathbf{x}; \tau, \delta_\tau]/J,
\label{eq varphi extension to non-free}
\end{equation}
for any two-sided ideal $ I \subseteq I^{\sigma, \delta_\sigma}(\mathbb{F}^n) $, where $ J = \varphi_A(I) \subseteq I^{\tau, \delta_\tau}(\mathbb{F}^n) $ is also a two-sided ideal. Furthermore, evaluations are also preserved by $ \varphi_A $ as in (\ref{eq varphi preserves evalautions}) over non-free multivariate skew polynomial rings.
\end{corollary}

We conclude with a remark on \textit{conjugacy} and the so-called product rule. Two affine points $ \mathbf{a}, \mathbf{b} \in \mathbb{F}^n $ are said to be $ (\tau,\delta_\tau) $-conjugate if there exists $ c \in \mathbb{F}^* $ such that
\begin{equation}
\mathbf{b} = \tau(c) \mathbf{a} c^{-1} + \delta_\tau(c) c^{-1}.
\label{eq conjugacion}
\end{equation}
This concept was introduced in \cite[Eq. (2.5)]{lam-leroy} when $ n = 1 $, and in \cite[Def. 11]{multivariateskew} in the general case. With notation as in Proposition \ref{prop conjugation of morphism and derivation}, it is easy to check that $ \mathbf{a} $ and $ \mathbf{b} $ are $ (\tau, \delta_\tau) $-conjugate if, and only if, $ A \mathbf{a} $ and $ A \mathbf{b} $ are $ (\sigma, \delta_\sigma) $-conjugate, with the same element $ c \in \mathbb{F}^* $ in the conjugacy relation (\ref{eq conjugacion}) in both cases.

The conjugacy relation allows to relate evaluations and products through the product rule, which was given in \cite[Th. 2.7]{lam-leroy} when $ n = 1 $, and in \cite[Th. 3]{multivariateskew} in the general case. This result is as follows. Let $ F(\mathbf{x}), G(\mathbf{x}) \in \mathbb{F}[\mathbf{x}; \tau, \delta_\tau] $, let $ \mathbf{a} \in \mathbb{F}^n $, and let $ c = G(\mathbf{a}) $. If $ c = 0 $, then $ (FG)(\mathbf{a}) = E^{\tau, \delta_\tau}_\mathbf{a}((FG)(\mathbf{x})) = 0 $, and if $ c \neq 0 $, then
$$ (FG)(\mathbf{a}) = E^{\tau, \delta_\tau}_\mathbf{a}((FG)(\mathbf{x})) = E^{\tau, \delta_\tau}_\mathbf{b}(F(\mathbf{x})) E^{\tau, \delta_\tau}_\mathbf{a}(G(\mathbf{x})) = F(\mathbf{b}) G(\mathbf{a}), $$
where $ \mathbf{a} $ and $ \mathbf{b} $ are $ (\tau, \delta_\tau) $-conjugate with $ c $ in the conjugacy relation (\ref{eq conjugacion}).

Therefore, Propositions \ref{prop phi A is ring morphism} and \ref{prop varphi preserves evaluations} are consistent with the product rule, since the following identities, which need to hold, actually hold:
$$ E^{\sigma,\delta_\sigma}_{A \mathbf{b}}(F(\mathbf{x})) E^{\sigma, \delta_\sigma}_{A \mathbf{a}}(G(\mathbf{x})) = E^{\sigma, \delta_\sigma}_{A\mathbf{a}}((FG)(\mathbf{x})) = E^{\tau,\delta_\tau}_{\mathbf{a}}((FG)(A \mathbf{x})) = $$
$$ E^{\tau,\delta_\tau}_{\mathbf{a}}(F(A\mathbf{x}) G(A\mathbf{x})) = E^{\tau,\delta_\tau}_{\mathbf{b}}(F(A\mathbf{x})) E^{\tau,\delta_\tau}_{\mathbf{a}}(G(A\mathbf{x})), $$
assuming that $ c = E^{\sigma, \delta_\sigma}_{A \mathbf{a}}(G(\mathbf{x})) = E^{\tau, \delta_\tau}_{\mathbf{a}}(G(A\mathbf{x})) \neq 0 $, where $ \mathbf{a} $ and $ \mathbf{b} $ are $ (\tau, \delta_\tau) $-conjugate, thus $ A\mathbf{a} $ and $ A \mathbf{b} $ are $ (\sigma, \delta_\sigma) $-conjugate, both using $ c $ in the conjugacy relation. Similarly if $ c = 0 $.

\subsection{The inverse of the map $ \varphi_A $}

In this subsection, we prove that the ring morphism $ \varphi_A $ has $ \varphi_{A^{-1}} $ as its inverse, and thus it is a ring isomorphism. The main result is the following.

\begin{proposition} \label{prop inverse of phi A}
Let $ \sigma, \tau, \upsilon : \mathbb{F} \longrightarrow \mathbb{F}^{n \times n} $ be ring morphisms and let $ \delta_\sigma, \delta_\tau, \delta_\upsilon : \mathbb{F} \longrightarrow \mathbb{F}^n $ be a $ \sigma $-derivation, a $ \tau $-derivation and a $ \upsilon $-derivation, respectively, where
$$ \sigma (a) = A \tau (a) A^{-1} \quad \textrm{and} \quad \delta_\sigma(a) = A \delta_\tau(a), $$
$$ \tau (a) = B \upsilon (a) B^{-1} \quad \textrm{and} \quad \delta_\tau(a) = B \delta_\upsilon(a), $$
for all $ a \in \mathbb{F} $, for invertible matrices $ A, B \in \mathbb{F}^{n \times n} $. It holds that
$$ \varphi_B \circ \varphi_A = \varphi_{AB}. $$
In other words, for any $ F(\mathbf{x}) \in \mathbb{F}[\mathbf{x}; \sigma, \delta_\sigma] $ and using notation as in (\ref{eq notation linear transform}), it holds that
$$ F(A(B \mathbf{x})) = F((AB)\mathbf{x}) $$
in the ring $ \mathbb{F}[\mathbf{x}; \upsilon, \delta_\upsilon] $. 
\end{proposition}
\begin{proof}
By the linearity of the maps $ \varphi_A $ and $ \varphi_B $, we only need to prove this result for $ F(\mathbf{x}) = \mathfrak{m}(\mathbf{x}) $, for $ \mathfrak{m}(\mathbf{x}) \in \mathcal{M} $. We do this recursively on $ \mathfrak{m}(\mathbf{x}) \in \mathcal{M} $.

First, the result is trivial for $ \mathfrak{m}(\mathbf{x}) = 1 $. Next, we have that
\begin{equation}
\varphi_B(\varphi_A(\mathbf{x})) = \varphi_B(A \mathbf{x}) = A(B\mathbf{x}) = (AB)\mathbf{x} = \varphi_{AB}(\mathbf{x}),
\label{eq proof evaluation for x}
\end{equation}
thus the result holds for $ \mathfrak{m}(\mathbf{x}) = x_i $, for $ i = 1,2, \ldots, n $. Finally, assume that the result holds for a given $ \mathfrak{m}(\mathbf{x}) \in \mathcal{M} $, and we will prove it for $ x_i \mathfrak{m}(\mathbf{x}) $, for $ i = 1,2, \ldots, n $. Let $ \mathbf{F}(\mathbf{x}) = \mathbf{x} \mathfrak{m}(\mathbf{x}) \in \mathbb{F}[\mathbf{x}; \sigma, \delta_\sigma]^n $ and denote $ \mathbf{G}(\mathbf{x}) = \mathbf{F}(A \mathbf{x}) \in \mathbb{F}[\mathbf{x}; \tau, \delta_\tau]^n $. By Definition \ref{def map phi A}, $ \mathbf{G}(\mathbf{x}) = (A \mathbf{x}) \mathfrak{m}(A \mathbf{x}) $. By (\ref{eq proof evaluation for x}) and the hypothesis on $ \mathfrak{m}(\mathbf{x}) $, we conclude that
$$ \varphi_B(\varphi_A(\mathbf{F}(\mathbf{x}))) = \mathbf{G}(B \mathbf{x}) = A((B \mathbf{x}) \mathfrak{m}((AB)\mathbf{x})) $$
$$ = ((AB) \mathbf{x}) \mathfrak{m}((AB)\mathbf{x}) = \mathbf{F}((AB)\mathbf{x}) = \varphi_{AB}(\mathbf{F}(\mathbf{x})), $$
and we are done.
\end{proof}

Moreover, we have the following obvious fact.

\begin{proposition}
Given a ring morphism $ \sigma : \mathbb{F} \longrightarrow \mathbb{F}^{n \times n} $ and a $ \sigma $-derivation $ \delta : \mathbb{F} \longrightarrow \mathbb{F}^n $, it holds that
$$ \varphi_I : \mathbb{F} [\mathbf{x} ; \sigma,\delta] \longrightarrow \mathbb{F} [\mathbf{x} ; \sigma,\delta] $$
is the identity morphism, where $ I \in \mathbb{F}^{n \times n} $ is the identity matrix.
\end{proposition}

Therefore, we conclude that $ \varphi_A $ is a ring isomorphism with inverse given by $ \varphi_{A^{-1}} $.

\begin{corollary} \label{cor varphi is isomorphism}
Let $ \sigma, \tau : \mathbb{F} \longrightarrow \mathbb{F}^{n \times n} $ be ring morphisms and let $ \delta_\sigma, \delta_\tau : \mathbb{F} \longrightarrow \mathbb{F}^n $ be a $ \sigma $-derivation and a $ \tau $-derivation, respectively, as in Proposition \ref{prop conjugation of morphism and derivation}, for an invertible matrix $ A \in \mathbb{F}^{n \times n} $. Observe that 
$$ \tau (a) = A^{-1} \sigma (a) A \quad \textrm{and} \quad \delta_\tau(a) = A^{-1} \delta_\sigma(a), $$
for all $ a \in \mathbb{F} $. We conclude that
$$ \varphi_A \circ \varphi_{A^{-1}} = \varphi_{A^{-1}} \circ \varphi_A = \varphi_I = {\rm Id}. $$
That is, $ \varphi_A^{-1} = \varphi_{A^{-1}} $, and $ \varphi_A $ is a ring isomorphism.
\end{corollary}

\begin{remark}
Observe that the results in this section hold automatically over non-free multivariate skew polynomial rings due to Corollary \ref{cor extend varphi to non-free}.
\end{remark}

\subsection{The map $ \varphi_A $ preserves degrees}

We conclude by showing that $ \varphi_A $ preserves degrees.

\begin{proposition} \label{prop varphi A preserves degrees}
Let $ \sigma, \tau : \mathbb{F} \longrightarrow \mathbb{F}^{n \times n} $ be ring morphisms and let $ \delta_\sigma, \delta_\tau : \mathbb{F} \longrightarrow \mathbb{F}^n $ be a $ \sigma $-derivation and a $ \tau $-derivation, respectively, as in Proposition \ref{prop conjugation of morphism and derivation}, for an invertible matrix $ A \in \mathbb{F}^{n \times n} $. It holds that
$$ \deg(\varphi_A(F(\mathbf{x}))) = \deg(F(\mathbf{x})), $$
for all $ F(\mathbf{x}) \in \mathbb{F}[\mathbf{x}; \sigma, \delta_\sigma] $.
\end{proposition}
\begin{proof}
From Definition \ref{def map phi A}, it is easy to check that $ \deg(\varphi_A(F(\mathbf{x}))) \leq \deg(F(\mathbf{x})) $, for all $ F(\mathbf{x}) \in \mathbb{F}[\mathbf{x}; \sigma, \delta_\sigma] $. Thus, it follows from Corollary \ref{cor varphi is isomorphism} that
$$ \deg(\varphi_A(F(\mathbf{x}))) \leq \deg(F(\mathbf{x})) = \deg(\varphi_{A^{-1}} (\varphi_A(F(\mathbf{x})))) \leq \deg(\varphi_A(F(\mathbf{x}))), $$
and the result follows.
\end{proof}

\section{Translations of variables} \label{sec translations}

In Subsection \ref{subsec derivations}, we showed that all derivations are inner derivations over finite fields (Theorem \ref{theorem description vector der}). 

In this section, we study translations of variables over multivariate skew polynomial rings with the same ring morphism $ \sigma : \mathbb{F} \longrightarrow \mathbb{F}^{n \times n} $ and whose $ \sigma $-derivations differ additively by an inner $ \sigma $-derivation. Such transformations of variables yield again ring isomorphisms between multivariate skew polynomial rings. By composing them with the linear transformations from the previous section, we will obtain in Section \ref{sec affine} a complete collection of affine transformations of variables between the corresponding multivariate skew polynomial rings. Combined with Theorem \ref{theorem description vector der}, these translations of variables will show in Section \ref{sec affine} how to simplify multivariate skew polynomial rings over finite fields.

Again, translations of variables preserve evaluations and degrees of free multivariate skew polynomials, hence they can be extended to \textit{non-free} multivariate skew polynomial rings. We will also divide the section in different subsections, each devoted to a different property of the ring isomorphism, which we now denote by $ \phi_{\boldsymbol\lambda} $ and which depends on the translation vector $ \boldsymbol\lambda \in \mathbb{F}^n $.

\subsection{Definition of the map $ \phi_{\boldsymbol\lambda} $} \label{subsec definition of phi}

As mentioned in Subsection \ref{subsec derivations}, the definition of inner derivations in Definition \ref{def vector der} gives derivations only over fields. A slight modification, first given in \cite[Example 2]{multivariateskew} when $ n > 1 $, allows to consider inner derivations over any division ring. 

\begin{definition} \label{def inner derivations}
Given a ring morphism $ \sigma : \mathbb{F} \longrightarrow \mathbb{F}^{n \times n} $, we say that $ \delta : \mathbb{F} \longrightarrow \mathbb{F}^n $ is an inner $ \sigma $-derivation if there exists $ \boldsymbol\lambda \in \mathbb{F}^n $ such that
$$ \delta(a) = \boldsymbol\lambda a - \sigma(a) \boldsymbol\lambda, $$
for all $ a \in \mathbb{F} $.
\end{definition}

Observe that, with this definition, an inner $ \sigma $-derivation is indeed a $ \sigma $-derivation, since it satisfies (\ref{eq derivations multiplicative property}) even if $ \mathbb{F} $ is non-commutative.

Translations of variables are then defined as follows.

\begin{definition} \label{def map phi lambda}
Fix a ring morphism $ \sigma : \mathbb{F} \longrightarrow \mathbb{F}^{n \times n} $ and $ \sigma $-derivations $ \delta, \delta^\prime : \mathbb{F} \longrightarrow \mathbb{F}^n $ such that $ \delta - \delta^\prime $ is an inner $ \sigma $-derivation, that is,
$$ \delta(a) - \delta^\prime(a) = \boldsymbol\lambda a - \sigma(a) \boldsymbol\lambda, $$
for $ a \in \mathbb{F} $, for a given $ \boldsymbol\lambda \in \mathbb{F}^n $. We then define the map
\begin{equation}
\phi_{\boldsymbol\lambda} : \mathbb{F}[\mathbf{x}; \sigma, \delta] \longrightarrow \mathbb{F}[\mathbf{x}; \sigma, \delta^\prime]
\label{eq definition of phi lambda}
\end{equation}
as follows. First, we define $ \phi_{\boldsymbol\lambda}(1) = 1 $ and 
$$ \phi_{\boldsymbol\lambda}(\mathbf{x}) = \left( \begin{array}{c}
\phi_{\boldsymbol\lambda} (x_1) \\
\phi_{\boldsymbol\lambda} (x_2) \\
\vdots \\
\phi_{\boldsymbol\lambda} (x_n) 
\end{array} \right) = \left( \begin{array}{c}
x_1 + \lambda_1 \\
x_2 + \lambda_2 \\
\vdots \\
x_n + \lambda_n 
\end{array} \right) = \mathbf{x} + \boldsymbol\lambda. $$
Next we define $ \phi_{\boldsymbol\lambda}(\mathfrak{m}(\mathbf{x})) $ recursively on monomials $ \mathfrak{m}(\mathbf{x}) \in \mathcal{M} $. Assume that $ \phi_{\boldsymbol\lambda}(\mathfrak{m}(\mathbf{x})) $ is defined, for a given $ \mathfrak{m}(\mathbf{x}) \in \mathcal{M} $. Then define
$$ \phi_{\boldsymbol\lambda}(\mathbf{x} \mathfrak{m}(\mathbf{x})) = (\mathbf{x} + \boldsymbol\lambda) \phi_{\boldsymbol\lambda}(\mathfrak{m}(\mathbf{x})) . $$
Finally, if $ F(\mathbf{x}) = \sum_{\mathfrak{m}(\mathbf{x}) \in \mathcal{M}} F_\mathfrak{m} \mathfrak{m}(\mathbf{x}) \in \mathbb{F}[\mathbf{x}; \sigma, \delta] $, where $ F_\mathfrak{m} \in \mathbb{F} $, for all $ \mathfrak{m}(\mathbf{x}) \in \mathcal{M} $, then we define
\begin{equation}
\phi_{\boldsymbol\lambda}(F(\mathbf{x})) = \sum_{\mathfrak{m}(\mathbf{x}) \in \mathcal{M}} F_\mathfrak{m} \phi_{\boldsymbol\lambda}(\mathfrak{m}(\mathbf{x})).
\label{eq def by linearity of phi lambda}
\end{equation}
Given $ F(\mathbf{x}) \in \mathbb{F}[\mathbf{x}; \sigma, \delta] $ or $ \mathfrak{m}(\mathbf{x}) \in \mathcal{M} $, we will use the notation
\begin{equation}
F(\mathbf{x} + \boldsymbol\lambda) = \phi_{\boldsymbol\lambda}(F(\mathbf{x})) \quad \textrm{and} \quad \mathfrak{m}(\mathbf{x} + \boldsymbol\lambda) = \phi_{\boldsymbol\lambda}(\mathfrak{m} (\mathbf{x})).
\label{eq notation translation}
\end{equation}
\end{definition}

\begin{remark}
Note that, if $ \mathbb{F} $ is a field and $ \sigma = {\rm Id} $, then $ \delta = \delta^\prime $. In particular, if $ \delta = \delta^\prime = 0 $, then both multivariate skew polynomial rings in (\ref{eq definition of phi lambda}) are the conventional free multivariate polynomial ring $ \mathbb{F}[\mathbf{x}] $. Moreover, in that case $ \phi_{\boldsymbol\lambda} $ coincides with the usual definition of $ F(\mathbf{x} + \boldsymbol\lambda) $ that consists in substituting $ \mathbf{x} $ by $ \mathbf{x} + \boldsymbol\lambda $, for any $ F(\mathbf{x}) \in \mathbb{F}[\mathbf{x}] $. Observe that if $ \mathbb{F} $ is non-commutative and $ \boldsymbol\lambda \neq \mathbf{0} $, then $ \delta \neq \delta^\prime $ since the map $ a \mapsto \boldsymbol\lambda a - a \boldsymbol\lambda $, for $ a \in \mathbb{F} $, is not necessarily zero if the components of $ \boldsymbol\lambda $ do not lie in the center of $ \mathbb{F} $.
\end{remark}

\subsection{The map $ \phi_{\boldsymbol\lambda} $ is additive and multiplicative}

In this subsection, we will prove that the map $ \phi_{\boldsymbol\lambda} $ in Definition \ref{def map phi lambda} is a ring morphism. The proofs of the properties of $ \phi_{\boldsymbol\lambda} $ are similar to those of $ \varphi_A $ in Section \ref{sec linear transformations}. However, for clarity, we include the proof of the following result.

\begin{proposition} \label{prop phi lambda is ring morphism}
Let $ \sigma : \mathbb{F} \longrightarrow \mathbb{F}^{n \times n} $ be a ring morphism and let $ \delta, \delta^\prime : \mathbb{F} \longrightarrow \mathbb{F}^n $ be $ \sigma $-derivations as in Definition \ref{def map phi lambda}, for a given $ \boldsymbol\lambda \in \mathbb{F}^n $. Then the map $ \phi_{\boldsymbol\lambda} : \mathbb{F}[\mathbf{x}; \sigma, \delta] \longrightarrow \mathbb{F}[\mathbf{x}; \sigma, \delta^\prime] $ in Definition \ref{def map phi lambda} is a ring morphism. That is, with notation as in (\ref{eq notation translation}), it holds that
$$ (F + G)(\mathbf{x} + \boldsymbol\lambda) = F(\mathbf{x} + \boldsymbol\lambda) + G(\mathbf{x} + \boldsymbol\lambda) \quad \textrm{and} \quad (FG)(\mathbf{x} + \boldsymbol\lambda) = F(\mathbf{x} + \boldsymbol\lambda) G(\mathbf{x} + \boldsymbol\lambda), $$
for all $ F(\mathbf{x}), G(\mathbf{x}) \in \mathbb{F}[\mathbf{x}; \sigma, \delta] $, where operations on the right-hand sides are in the ring $ \mathbb{F}[\mathbf{x}; \sigma, \delta^\prime] $, and operations on the left-hand sides are in the ring $ \mathbb{F}[\mathbf{x}; \sigma, \delta] $.
\end{proposition}
\begin{proof}
The fact that $ \phi_{\boldsymbol\lambda} $ is additive (or even left linear over $ \mathbb{F} $) follows directly from the definitions.

Because of (\ref{eq def by linearity of phi lambda}), to prove that $ \phi_{\boldsymbol\lambda} $ is multiplicative, we only need to show that
$$ \phi_{\boldsymbol\lambda}(\mathfrak{m}(\mathbf{x}) G(\mathbf{x})) = \phi_{\boldsymbol\lambda}(\mathfrak{m}(\mathbf{x})) \phi_{\boldsymbol\lambda}(G(\mathbf{x})), $$
for all $ \mathfrak{m}(\mathbf{x}) \in \mathcal{M} $ and all $ G(\mathbf{x}) \in \mathbb{F}[\mathbf{x}; \sigma, \delta] $. We do this recursively on $ \mathfrak{m}(\mathbf{x}) \in \mathcal{M} $.

First, the result is trivial for $ \mathfrak{m}(\mathbf{x}) = 1 $. Following the proof of Proposition \ref{prop phi A is ring morphism}, we need to prove the result for $ \mathfrak{m}(\mathbf{x}) = x_i $, for all $ i = 1,2, \ldots, n $, to perform the induction step. Denote $ G(\mathbf{x}) = \sum_{\mathfrak{m}(\mathbf{x}) \in \mathcal{M}} G_\mathfrak{m} \mathfrak{m}(\mathbf{x}) $, where $ G_\mathfrak{m} \in \mathbb{F} $, for all $ \mathfrak{m}(\mathbf{x}) \in \mathcal{M} $. First, it holds that
$$ \mathbf{x} G(\mathbf{x}) = \sum_{\mathfrak{m}(\mathbf{x}) \in \mathcal{M}} \mathbf{x}(G_\mathfrak{m} \mathfrak{m}(\mathbf{x})) = \sum_{\mathfrak{m}(\mathbf{x}) \in \mathcal{M}} (\sigma(G_\mathfrak{m}) \mathbf{x} \mathfrak{m}(\mathbf{x}) + \delta(G_\mathfrak{m}) \mathfrak{m}(\mathbf{x}) ), $$
since these operations are in the ring $ \mathbb{F}[\mathbf{x}; \sigma, \delta] $. Thus by definition, it holds that
$$ \phi_{\boldsymbol\lambda}( \mathbf{x} G(\mathbf{x})) = \sum_{\mathfrak{m}(\mathbf{x}) \in \mathcal{M}} (\sigma(G_\mathfrak{m}) (\mathbf{x} + \boldsymbol\lambda ) \mathfrak{m}(\mathbf{x} + \boldsymbol\lambda) + \delta(G_\mathfrak{m}) \mathfrak{m}(\mathbf{x} + \boldsymbol\lambda ) ). $$
Next, working in the ring $ \mathbb{F}[\mathbf{x}; \sigma, \delta^\prime] $, we have that
$$ \phi_{\boldsymbol\lambda}( \mathbf{x}) \phi_{\boldsymbol\lambda}(G(\mathbf{x})) = (\mathbf{x} + \boldsymbol\lambda) G(\mathbf{x} + \boldsymbol\lambda) = $$
$$ \sum_{\mathfrak{m}(\mathbf{x}) \in \mathcal{M}} (\sigma(G_\mathfrak{m}) \mathbf{x} \mathfrak{m}(\mathbf{x} + \boldsymbol\lambda) + \delta^\prime(G_\mathfrak{m}) \mathfrak{m}(\mathbf{x} + \boldsymbol\lambda) + \boldsymbol\lambda G_\mathfrak{m} \mathfrak{m}(\mathbf{x} + \boldsymbol\lambda) ) = $$
$$ \sum_{\mathfrak{m}(\mathbf{x}) \in \mathcal{M}} (\sigma(G_\mathfrak{m}) (\mathbf{x} + \boldsymbol\lambda) \mathfrak{m}(\mathbf{x} + \boldsymbol\lambda) + \delta^\prime(G_\mathfrak{m}) \mathfrak{m}(\mathbf{x} + \boldsymbol\lambda) + (\boldsymbol\lambda G_\mathfrak{m} - \sigma(G_\mathfrak{m}) \boldsymbol\lambda) \mathfrak{m}(\mathbf{x} + \boldsymbol\lambda) ) = $$
$$ \sum_{\mathfrak{m}(\mathbf{x}) \in \mathcal{M}} (\sigma(G_\mathfrak{m}) (\mathbf{x} + \boldsymbol\lambda) \mathfrak{m}(\mathbf{x} + \boldsymbol\lambda) + \delta(G_\mathfrak{m}) \mathfrak{m}(\mathbf{x} + \boldsymbol\lambda) ) . $$
Therefore, we deduce that
\begin{equation}
\phi_{\boldsymbol\lambda}(\mathbf{x} G(\mathbf{x})) = \phi_{\boldsymbol\lambda}( \mathbf{x}) \phi_{\boldsymbol\lambda}(G(\mathbf{x})),
\label{eq proof phi lambda morphism}
\end{equation}
where operations on the right-hand side are in the ring $ \mathbb{F}[\mathbf{x}; \sigma, \delta^\prime] $.

Finally, assume that the result holds for a given $ \mathfrak{m}(\mathbf{x}) \in \mathcal{M} $. To prove it for $ x_i \mathfrak{m}(\mathbf{x}) $, for $ i = 1,2, \ldots, n $, we may associate terms as in the proof of Proposition \ref{prop phi A is ring morphism}, using now (\ref{eq proof phi lambda morphism}), and we are done.
\end{proof}

\subsection{The map $ \phi_{\boldsymbol\lambda} $ preserves evaluations}

In this subsection, we show that the $ (\sigma, \delta^\prime) $-evaluation of $ F(\mathbf{x} + \boldsymbol\lambda) \in \mathbb{F}[\mathbf{x}; \sigma, \delta^\prime] $ at an affine point $ \mathbf{a} \in \mathbb{F}^n $ coincides with the $ (\sigma, \delta) $-evaluation of $ F(\mathbf{x}) \in \mathbb{F}[\mathbf{x}; \sigma, \delta] $ at the affine point $ \mathbf{a} + \boldsymbol\lambda \in \mathbb{F}^n $. See (\ref{eq evaluation}) in Section \ref{sec preliminaries} for the definition of evaluation.

\begin{proposition} \label{prop phi lambda evaluations}
Let $ \sigma : \mathbb{F} \longrightarrow \mathbb{F}^{n \times n} $ be a ring morphism and let $ \delta, \delta^\prime : \mathbb{F} \longrightarrow \mathbb{F}^n $ be $ \sigma $-derivations as in Definition \ref{def map phi lambda}, for a given $ \boldsymbol\lambda \in \mathbb{F}^n $. The map $ \phi_{\boldsymbol\lambda} : \mathbb{F}[\mathbf{x}; \sigma, \delta] \longrightarrow \mathbb{F}[\mathbf{x}; \sigma, \delta^\prime] $ in Definition \ref{def map phi lambda} preserves evaluations over any point $ \mathbf{a} \in \mathbb{F}^n $ after translation by the vector $ \boldsymbol\lambda $. More concretely, according to the definition of evaluation in (\ref{eq evaluation}), it holds that
$$ E_{\mathbf{a}}^{\sigma, \delta^\prime}(F(\mathbf{x} + \boldsymbol\lambda)) = E_{\mathbf{a} + \boldsymbol\lambda}^{\sigma, \delta}(F(\mathbf{x})), $$
for all $ F(\mathbf{x}) \in \mathbb{F}[\mathbf{x}; \sigma, \delta] $ and all $ \mathbf{a} \in \mathbb{F}^n $. In other words, the evaluation of $ F(\mathbf{x} + \boldsymbol\lambda) $ in the point $ \mathbf{a} $ is $ F(\mathbf{a} + \boldsymbol\lambda) $.
\end{proposition}
\begin{proof}
Analogous to the proof of Proposition \ref{prop varphi preserves evaluations}.
\end{proof}

Again, using this result, we may extend $ \phi_{\boldsymbol\lambda} $ to non-free multivariate skew polynomial rings, following the definitions in Section \ref{sec preliminaries} and the notation in Corollary \ref{cor extend varphi to non-free}. 

\begin{corollary} \label{cor extend phi lambda to non-free}
Let $ \sigma : \mathbb{F} \longrightarrow \mathbb{F}^{n \times n} $ be a ring morphism and let $ \delta, \delta^\prime : \mathbb{F} \longrightarrow \mathbb{F}^n $ be $ \sigma $-derivations as in Definition \ref{def map phi lambda}, for a given $ \boldsymbol\lambda \in \mathbb{F}^n $. It holds that
$$ \phi_{\boldsymbol\lambda} \left( I^{\sigma, \delta}(\mathbb{F}^n) \right) = I^{\sigma, \delta^\prime}(\mathbb{F}^n). $$
Therefore, we may extend $ \phi_{\boldsymbol\lambda} $ to a ring morphism between non-free multivariate skew polynomial rings, that is, to ring morphisms
\begin{equation}
\phi_{\boldsymbol\lambda} : \mathbb{F}[\mathbf{x}; \sigma, \delta]/I \longrightarrow \mathbb{F}[\mathbf{x}; \sigma, \delta^\prime]/J,
\label{eq phi lambda extension to non-free}
\end{equation}
for any two-sided ideal $ I \subseteq I^{\sigma, \delta}(\mathbb{F}^n) $, where $ J = \phi_{\boldsymbol\lambda}(I) \subseteq I^{\sigma, \delta^\prime}(\mathbb{F}^n) $ is also a two-sided ideal. Furthermore, evaluations are also preserved by $ \phi_{\boldsymbol\lambda} $ as in (\ref{eq varphi preserves evalautions}) over non-free multivariate skew polynomial rings.
\end{corollary}

Finally, the remarks on conjugacy and the product rule after Corollary \ref{cor extend varphi to non-free} can be translated mutatis mutandis to translations of variables. We leave the details to the reader, but we remark that, in this case, $ \mathbf{a} $ and $ \mathbf{b} $ are $ (\sigma, \delta^\prime) $-conjugate if, and only if, $ \mathbf{a} + \boldsymbol\lambda $ and $ \mathbf{b} + \boldsymbol\lambda $ are $ (\sigma, \delta) $-conjugate, with the same element $ c \in \mathbb{F}^* $ in the conjugacy relation (\ref{eq conjugacion}) in both cases.

\subsection{The inverse of the map $ \phi_{\boldsymbol\lambda} $}

In this subsection, we show that the ring morphism $ \phi_{\boldsymbol\lambda} $ has $ \phi_{- \boldsymbol\lambda} $ as its inverse, and thus it is a ring isomorphism. The main result is the following.

\begin{proposition} \label{prop inverse of phi lambda}
Let $ \sigma : \mathbb{F} \longrightarrow \mathbb{F}^{n \times n} $ be a ring morphism and let $ \delta, \delta^\prime, \delta^{\prime \prime} : \mathbb{F} \longrightarrow \mathbb{F}^n $ be $ \sigma $-derivations such that
$$ \delta(a) - \delta^\prime(a) = \boldsymbol\lambda a - \sigma(a) \boldsymbol\lambda \quad \textrm{and} \quad \delta^\prime(a) - \delta^{\prime \prime}(a) = \boldsymbol\lambda^\prime a - \sigma(a) \boldsymbol\lambda^\prime, $$
for all $ a \in \mathbb{F} $, for certain $ \boldsymbol\lambda, \boldsymbol\lambda^\prime \in \mathbb{F}^n $. It holds that
$$ \phi_{\boldsymbol\lambda^\prime} \circ \phi_{\boldsymbol\lambda} = \phi_{\boldsymbol\lambda + \boldsymbol\lambda^\prime}. $$
In addition, we have that $ \phi_\mathbf{0} = {\rm Id} $, hence
$$ \phi_{\boldsymbol\lambda} \circ \phi_{-\boldsymbol\lambda} = \phi_{-\boldsymbol\lambda} \circ \phi_{\boldsymbol\lambda} = \phi_\mathbf{0} = {\rm Id}. $$ 
That is, $ \phi_{\boldsymbol\lambda}^{-1} = \phi_{-\boldsymbol\lambda} $, and $ \phi_{\boldsymbol\lambda} $ is a ring isomorphism. 
\end{proposition}
\begin{proof}
The proof follows the same lines as in the proof of Proposition \ref{prop inverse of phi A}.
\end{proof}

\begin{remark}
Again, the results in this section hold automatically over non-free multivariate skew polynomial rings due to Corollary \ref{cor extend phi lambda to non-free}.
\end{remark}

\subsection{The map $ \phi_{\boldsymbol\lambda} $ preserves degrees}

As it was the case for linear transformation of variables, the map $ \phi_{\boldsymbol\lambda} $ preserves degrees, which follows immediately from Proposition \ref{prop inverse of phi lambda}, analogously to the proof of Proposition \ref{prop varphi A preserves degrees}.
 
\begin{proposition} \label{prop phi lambda preserves degrees}
Let $ \sigma : \mathbb{F} \longrightarrow \mathbb{F}^{n \times n} $ be a ring morphism and let $ \delta, \delta^\prime : \mathbb{F} \longrightarrow \mathbb{F}^n $ be $ \sigma $-derivations as in Definition \ref{def map phi lambda}, for a given $ \boldsymbol\lambda \in \mathbb{F}^n $. It holds that
$$ \deg(\phi_{\boldsymbol\lambda}(F(\mathbf{x}))) = \deg(F(\mathbf{x})), $$
for all $ F(\mathbf{x}) \in \mathbb{F}[\mathbf{x}; \sigma, \delta] $.
\end{proposition}

\section{Affine transformations of variables} \label{sec affine}

This section is devoted to defining general affine transformations of variables between the corresponding multivariate skew polynomial rings and showing that they are the only ring isomorphisms that preserve degrees.

We start with the following observation, whose proof is left to the reader.

\begin{proposition} \label{prop commutativity linear and translations}
Let $ \sigma, \tau : \mathbb{F} \longrightarrow \mathbb{F}^{n \times n} $ be ring morphisms and let $ \delta_\sigma, \delta_\tau : \mathbb{F} \longrightarrow \mathbb{F}^n $ be a $ \sigma $-derivation and a $ \tau $-derivation, respectively, such that
$$ \sigma (a) = A \tau (a) A^{-1} \quad \textrm{and} \quad \delta_\sigma(a) = A \delta_\tau(a), $$
for all $ a \in \mathbb{F} $, for some invertible matrix $ A \in \mathbb{F}^{n \times n} $. Next, let $ \delta_\tau^\prime : \mathbb{F} \longrightarrow \mathbb{F}^n $ be a $ \tau $-derivation such that
$$ \delta_\tau(a) - \delta_\tau^\prime(a) = \boldsymbol\lambda a - \tau(a) \boldsymbol\lambda, $$
for all $ a \in \mathbb{F} $, for some $ \boldsymbol\lambda \in \mathbb{F}^n $. Then it holds that
$$ \phi_{\boldsymbol\lambda} \circ \varphi_A = \varphi_A \circ \phi_{(A^{-1} \boldsymbol\lambda)}. $$
In other words, and using the notations in (\ref{eq notation linear transform}) and (\ref{eq notation translation}), we have that
$$ F((A \mathbf{x}) + \boldsymbol\lambda) = F(A(\mathbf{x} + A^{-1} \boldsymbol\lambda)), $$
for all $ F(\mathbf{x}) \in \mathbb{F}[\mathbf{x}; \sigma, \delta_\sigma] $, in the ring $ \mathbb{F}[\mathbf{x}; \tau, \delta_\tau^\prime] $.
\end{proposition}

Thus we deduce the following.

\begin{corollary} \label{cor linear transform and translations commute}
Any composition in any order of a finite collection of linear transformations and translations of variables over multivariate skew polynomial rings is the composition of one linear transformation and one translation.
\end{corollary}

This corollary motivates the following definition of affine transformations of variables over multivariate skew polynomial rings.

\begin{definition} \label{def affine transformations}
Let $ \sigma, \tau : \mathbb{F} \longrightarrow \mathbb{F}^{n \times n} $ and $ \delta_\sigma, \delta_\tau, \delta_\tau^\prime : \mathbb{F} \longrightarrow \mathbb{F}^n $ be as in Proposition \ref{prop commutativity linear and translations}. An affine transformation of variables between the multivariate skew polynomial rings $ \mathbb{F}[\mathbf{x}; \sigma, \delta_\sigma] $ and $ \mathbb{F}[\mathbf{x}; \tau, \delta_\tau^\prime] $ is a ring isomorphism of the form 
$$ \mathcal{T}_{A, \boldsymbol\lambda} = \phi_{\boldsymbol\lambda} \circ \varphi_A : \mathbb{F}[\mathbf{x}; \sigma, \delta_\sigma] \longrightarrow \mathbb{F}[\mathbf{x}; \tau, \delta_\tau^\prime]. $$
\end{definition}

Corollary \ref{cor linear transform and translations commute} also implies that the composition of affine transformations is again an affine transformation. However, the domains and codomains of such maps are not all the same when $ (\sigma,\delta_\sigma) \neq ({\rm Id},0) $ or $ \mathbb{F} $ is not commutative.

To conclude, we show that affine transformations constitute all left $ \mathbb{F} $-linear ring isomorphisms (thus left $ \mathbb{F} $-algebra isomorphisms) between multivariate skew polynomial rings that preserve degrees.

\begin{theorem} \label{th affine transformations}
Let $ \sigma, \tau : \mathbb{F} \longrightarrow \mathbb{F}^{n \times n} $ be ring morphisms and let $ \delta_\sigma, \delta_\tau^\prime : \mathbb{F} \longrightarrow \mathbb{F}^n $ be a $ \sigma $-derivation and a $ \tau $-derivation, respectively. A map $ \mathcal{T} : \mathbb{F}[\mathbf{x}; \sigma, \delta_\sigma] \longrightarrow \mathbb{F}[\mathbf{x}; \tau, \delta_\tau^\prime] $ is a left $ \mathbb{F} $-linear ring isomorphism such that $ \deg(\mathcal{T}(F(\mathbf{x}))) = \deg(F(\mathbf{x})) $, for all $ F(\mathbf{x}) \in \mathbb{F}[\mathbf{x}; \sigma, \delta_\sigma] $, if and only if, there exists a $ \tau $-derivation $ \delta_\tau : \mathbb{F} \longrightarrow \mathbb{F}^n $, an invertible matrix $ A \in \mathbb{F}^{n \times n} $ and a vector $ \boldsymbol\lambda \in \mathbb{F}^n $, all as in Proposition \ref{prop commutativity linear and translations}, such that $ \mathcal{T} = \mathcal{T}_{A, \boldsymbol\lambda} $.
\end{theorem}
\begin{proof}
The direct implication follows from Definition \ref{def affine transformations} and Propositions \ref{prop phi A is ring morphism}, \ref{prop varphi A preserves degrees}, \ref{prop phi lambda is ring morphism} and \ref{prop phi lambda preserves degrees}. For the reversed implication, the preservation of degrees for the monomials $ x_1, x_2, \ldots, x_n $ implies that there exist an invertible matrix $ A \in \mathbb{F}^{n \times n} $ and a vector $ \boldsymbol\lambda \in \mathbb{F}^n $ such that
$$ \mathcal{T}(\mathbf{x}) = A \mathbf{x} + \boldsymbol\lambda. $$
Let $ a \in \mathbb{F} $. In the ring $ \mathbb{F}[\mathbf{x}; \sigma, \delta_\sigma] $, we have that $ \mathbf{x} a = \sigma(a) \mathbf{x} + \delta_\sigma(a) $ by (\ref{eq def inner product}). Hence 
\begin{equation}
\mathcal{T}(\mathbf{x} a) = \mathcal{T}(\sigma(a) \mathbf{x} + \delta_\sigma(a)) = \sigma(a) \mathcal{T}(\mathbf{x}) + \delta_\sigma(a) = \sigma(a) A \mathbf{x} + \sigma(a) \boldsymbol\lambda + \delta_\sigma(a).
\label{eq proof preserve deg are affine 1}
\end{equation}
Next, since $ \mathcal{T} $ is a left $ \mathbb{F} $-linear ring isomorphism, we also have that
\begin{equation}
\mathcal{T}(\mathbf{x} a) = \mathcal{T}(\mathbf{x}) \mathcal{T}(a1) = \mathcal{T}(\mathbf{x}) a = A \mathbf{x} a + \boldsymbol\lambda a = A \tau(a) \mathbf{x} + A \delta_\tau^\prime(a) + \boldsymbol\lambda a.
\label{eq proof preserve deg are affine 2}
\end{equation}
Comparing (\ref{eq proof preserve deg are affine 1}) and (\ref{eq proof preserve deg are affine 2}), the reader may check that there exists a $ \tau $-derivation $ \delta_\tau : \mathbb{F} \longrightarrow \mathbb{F}^n $, namely $ \delta_\tau = A^{-1} \delta_\sigma $, such that $ \sigma $, $ \tau $, $ \delta_\sigma $, $ \delta_\tau $ and $ \delta_\tau^\prime $ are all related as in Proposition \ref{prop commutativity linear and translations} for the matrix $ A \in \mathbb{F}^{n \times n} $ and the vector $ \boldsymbol\lambda \in \mathbb{F}^n $.

Finally, proceeding recursively on monomials $ \mathfrak{m}(\mathbf{x}) \in \mathcal{M} $ as in the proofs of Propositions \ref{prop phi A is ring morphism} and \ref{prop phi lambda is ring morphism}, the reader may check that $ \mathcal{T} = \phi_{\boldsymbol\lambda} \circ \varphi_A = \mathcal{T}_{A, \boldsymbol\lambda} $, and we are done.
\end{proof}

\section{Classification over finite fields} \label{sec classification finite fields}

With all the tools gathered up to this point, we are ready to show that all multivariate skew polynomial rings over a finite field $ \mathbb{F}_q $ are naturally isomorphic as $ \mathbb{F}_q $-algebras to some multivariate skew polynomial ring of the form $ \mathbb{F}_q[\mathbf{x}; {\rm diag}(\sigma_1, \sigma_2, \ldots, \sigma_n), 0] $ (Theorem \ref{th reduction finite fields}). Even further, two such representations only differ in a permutation of the field automorphisms $ \sigma_1, \sigma_2, \ldots, \sigma_n : \mathbb{F}_q \longrightarrow \mathbb{F}_q $, providing a full classification of multivariate skew polynomial rings over finite fields (Theorem \ref{th classification}). 

For ease of notation, given field automorphisms $ \sigma_1, \sigma_2, \ldots, \sigma_n : \mathbb{F}_q \longrightarrow \mathbb{F}_q $, we will denote
\begin{equation}
\mathbb{F}_q[\mathbf{x}; \sigma_1, \sigma_2, \ldots, \sigma_n] = \mathbb{F}_q[\mathbf{x}; {\rm diag}(\sigma_1, \sigma_2, \ldots, \sigma_n), 0].
\label{eq def diagonal ring}
\end{equation}

The next result extends the discussion given in \cite[Sec. 8.3]{cohn} (see also \cite[Prop. 40]{linearizedRS}) from the case $ n = 1 $ to the general case.

\begin{theorem} \label{th reduction finite fields}
Let $ \sigma : \mathbb{F}_q \longrightarrow \mathbb{F}_q^{n \times n} $ be a ring morphism and let $ \delta : \mathbb{F}_q \longrightarrow \mathbb{F}_q^n $ be a $ \sigma $-derivation. By Theorem \ref{theorem description matrix morph}, there exist an invertible matrix $ A \in \mathbb{F}_q^{n \times n} $ and field automorphisms $ \sigma_1, \sigma_2, \ldots, \sigma_n : \mathbb{F}_q \longrightarrow \mathbb{F}_q $ such that
$$ \sigma(a) = A \tau(a) A^{-1}, $$
for all $ a \in \mathbb{F}_q $, where $ \tau = {\rm diag}(\sigma_1, \sigma_2, \ldots, \sigma_n) $. Since the map $ A^{-1} \delta $ is a $ \tau $-derivation by Proposition \ref{prop conjugation of morphism and derivation}, we deduce from Theorem \ref{theorem description vector der} that there exists a vector $ \boldsymbol\lambda \in \mathbb{F}_q^n $ such that
$$ \delta(a) = A(\boldsymbol\lambda a - \tau(a) \boldsymbol\lambda), $$
for all $ a \in \mathbb{F}_q $. Therefore, by the results in Sections \ref{sec linear transformations} and \ref{sec translations}, the affine transformation of variables (Section \ref{sec affine})
$$ \mathcal{T}_{A, \boldsymbol\lambda} = \phi_{\boldsymbol\lambda} \circ \varphi_A : \mathbb{F}_q[\mathbf{x}; \sigma, \delta] \longrightarrow \mathbb{F}_q[\mathbf{x}; \sigma_1, \sigma_2, \ldots, \sigma_n] $$
is an $ \mathbb{F}_q $-algebra isomorphism that preserves evaluations (see Propositions \ref{prop varphi preserves evaluations} and \ref{prop phi lambda evaluations}) and degrees (see Propositions \ref{prop varphi A preserves degrees} and \ref{prop phi lambda preserves degrees}) and that may naturally be extended to non-free multivariate skew polynomial rings (see Corollaries \ref{cor extend varphi to non-free} and \ref{cor extend phi lambda to non-free}).
\end{theorem}

Theorem \ref{th reduction finite fields} says that we may simplify $ \mathbb{F}_q[\mathbf{x}; \sigma, \delta] $ to $ \mathbb{F}_q[\mathbf{x}; \sigma_1, \sigma_2, \ldots, \sigma_n] $ via an affine transformation of variables. The next theorem shows that two such reductions only differ by a permutation of $ \sigma_1, \sigma_2, \ldots, \sigma_n $.

\begin{theorem} \label{th classification}
Let $ \sigma_1, \sigma_2, \ldots, \sigma_n, \tau_1, \tau_2, \ldots, \tau_n : \mathbb{F}_q \longrightarrow \mathbb{F}_q $ be field automorphisms. There exists an affine transformation
$$ \mathcal{T} : \mathbb{F}_q[\mathbf{x}; \sigma_1, \sigma_2, \ldots, \sigma_n] \longrightarrow \mathbb{F}_q[\mathbf{x}; \tau_1, \tau_2, \ldots, \tau_n] $$
if, and only if, there exists a bijection $ s : \{ 1,2, \ldots, n \} \longrightarrow \{ 1,2, \ldots, n \} $ such that $ \tau_i = \sigma_{s(i)} $, for $ i = 1,2, \ldots, n $.
\end{theorem}
\begin{proof}
The reversed implication is trivial. We now prove the direct implication. If there exists such an affine transformation, then by Theorem \ref{th affine transformations}, there exists an invertible matrix $ A \in \mathbb{F}_q^{n \times n} $ such that
$$ \left( \begin{array}{cccc}
\sigma_1(c) & 0 & \ldots & 0 \\
0 & \sigma_2(c) & \ldots & 0 \\
\vdots & \vdots & \ddots & \vdots \\
0 & 0 & \ldots & \sigma_n(c)
\end{array} \right) = A \left( \begin{array}{cccc}
\tau_1(c) & 0 & \ldots & 0 \\
0 & \tau_2(c) & \ldots & 0 \\
\vdots & \vdots & \ddots & \vdots \\
0 & 0 & \ldots & \tau_n(c)
\end{array} \right) A^{-1}, $$
where $ c \in \mathbb{F}_q^* $ is a primitive element. Since the eigenvalues of a square matrix are invariant by similarity of matrices, we deduce that there exists a bijection $ s : \{ 1,2, \ldots, n \} \longrightarrow \{ 1,2, \ldots, n \} $ such that $ \tau_i(c) = \sigma_{s(i)}(c) $, for $ i = 1,2, \ldots, n $. Now, if $ \tau_i(c) = \sigma_{s(i)}(c) $, then $ \tau_i(a) = \sigma_{s(i)}(a) $, for all $ a \in \mathbb{F}_q $, since $ c \in \mathbb{F}_q^* $ is a primitive element, and we are done.
\end{proof}

In conclusion, if we identify multivariate skew polynomial rings that are isomorphic by an $ \mathbb{F}_q $-algebra isomorphism that preserves degrees (i.e., by an affine transformation of variables), then such classes are represented by the rings $ \mathbb{F}_q[\mathbf{x}; \sigma_1, \sigma_2, \ldots, \sigma_n] $ modulo permutations of the field automorphisms $ \sigma_1, \sigma_2, \ldots, \sigma_n $.

Finally, Theorems \ref{theorem description vector der}, \ref{theorem description matrix morph} and \ref{th reduction finite fields} give an explicit method to find the reduced form $ \mathbb{F}_q[\mathbf{x}; \sigma_1, \sigma_2, \ldots, \sigma_n] $ from $ \sigma(c) \in \mathbb{F}_q^{n \times n} $ and $ \delta(c) \in \mathbb{F}_q^n $, for a primitive element $ c \in \mathbb{F}_q^* $.

\section*{Acknowledgement}

The author gratefully acknowledges the support from The Independent Research Fund Denmark (Grant No. DFF-7027-00053B).

%
%
%
%
%
%


\end{document}